\documentclass[11pt]{article}

\usepackage{mathrsfs,bbm,bm,enumitem, diagbox}
\usepackage[margin=1in]{geometry}

\newcommand{\BP}{\mathbb{P}}

\newcommand{\E}{\mathbb{E}}
\newcommand{\Id}{\mathrm{Id}}

\newcommand{\les}{\lesssim}

\newcommand{\norm}[1]{\left\|#1\right\|}

\newcommand{\TV}{\mathrm{TV}}

\usepackage{amsmath,graphicx, amsthm, amssymb, epsfig}
\usepackage{xcolor}
\usepackage[normalem]{ulem}
\numberwithin{equation}{section}

\usepackage{hyperref}
\hypersetup{
    colorlinks=true,
    linkcolor=blue,
    filecolor=magenta,      
    urlcolor=cyan,
    citecolor=teal,
    pdfpagemode=FullScreen,
}

\newcommand{\e}{\epsilon}
\newcommand{\ga}{\gamma}
\newcommand{\de}{\delta}
\newcommand{\br}{\mathbb{R}}
\newcommand{\N}{\mathbb{N}}
\newcommand{\BZ}{\mathbb{Z}}

\newcommand{\pa}{\partial}

\newcommand{\al}{\alpha}
\newcommand{\la}{\lambda}

\newcommand{\be}{\begin{equation}}
\newcommand{\ee}{\end{equation}}
\def\bs#1\es{
    \begin{equation}\begin{split}
    #1
    \end{split}\end{equation}
}
\def\bsn#1\esn{
    \begin{equation*}\begin{split}
    #1
    \end{split}\end{equation*}
}

\newcommand{\dd}{\mathrm{d}}
\newcommand{\op}{\mathrm{op}}

\newcommand{\bma}{\begin{pmatrix}}
\newcommand{\ema}{\end{pmatrix}}

\newcommand{\I}{\mathcal  I}

\newcommand{\CK}{\mathcal{K}}
\newcommand{\CL}{\mathcal{L}}
\newcommand{\CN}{\mathcal{N}}

\newcommand{\tsum}{\textstyle\sum}

\newtheorem{theorem}{Theorem}[section]
\newtheorem{lemma}[theorem]{Lemma}

\theoremstyle{definition}
\newtheorem{assumption}[theorem]{Assumption}

\begin{document}

\title{Discrete to continuum limits in Bayesian inverse problems}
\author{Alexander Katsevich\\
School of Data, Mathematical, and Statistical Sciences,\\ University of Central Florida\\ \texttt{alexander.katsevich@ucf.edu}}

\date{}
\maketitle

\newcommand{\kind}{\mathcal J_N}
\newcommand{\Psif}{\Psi^{\mathrm{full}}}
\newcommand{\Psinf}{\Psi_N^{\mathrm{full}}}

\begin{abstract}
We develop a posterior-level discrete-to-continuum theory for Bayesian inverse problems 
whose finite-dimensional priors arise from local finite-difference regularization and whose likelihoods are of a general convex GLM-type form. In contrast to continuum-first approaches, the continuum prior and posterior are not assumed at the outset but are constructed as limits of the finite-dimensional probability measures used in numerical computation.

First, with the total information parameter $\tau$ and the regularization strength $\kappa$ fixed, we prove weak convergence on $L^2$ of the reconstructed discrete Gaussian priors and posterior measures to well-defined continuum laws. We then consider a coupled grid-refinement and small-noise limit in which $N\to\infty$ and $\tau_N,\kappa_N\to\infty$, while $\kappa_N/\tau_N$ remains fixed. Under explicit growth conditions relating $N$ and $\tau_N$, we prove that the reconstructed discrete MAP estimates converge to the unique minimizer $u_*$ of the limiting continuum cost functional, that the reconstructed posterior measures concentrate at $u_*$, and that their centered and rescaled fluctuations converge to $\mathcal N(\vec 0,Q^{-1})$, where $Q$ is the Hessian of the continuum cost functional at $u_*$. Finally, we show that the same deterministic limit and Gaussian fluctuation law are obtained by first constructing the continuum posterior and then taking its small-noise, high-regularization limit.
\end{abstract}

\section{Introduction}

An inverse problem seeks to reconstruct a discrete approximation $\vec f_N\in\br^N$ to an unknown object $f$ from $N$ noisy measurements of its image, $\CK f$, under a forward operator $\CK$. We assume that $\CK$ is linear. In the Bayesian formulation, one specifies a prior probability law $\nu_N$ for $\vec f_N$ and a likelihood for the observed data. If the negative log-likelihood is represented by a potential $\Psi_N$, the posterior measure has the structure
\be
\dd\mu_N(\vec f_N)=Z_N^{-1}\exp\{-\Psi_N(\vec f_N)\}\,\dd\nu_N(\vec f_N),
\ee
where $Z_N$ is the normalizing constant. The dependence of $\Psi_N$ on the observed data is suppressed for simplicity.

As is customary, by solving a Bayesian inverse problem (BIP), we mean obtaining an accurate approximation of the posterior measure $\mu_N$, rather than merely computing a maximum a posteriori (MAP) estimate \cite{cotter2010}. Although every numerical implementation of a BIP is finite-dimensional, the resulting problems are often very high-dimensional. Their formulation and analysis in an infinite-dimensional setting offer numerous advantages \cite[Introduction]{Dashti2017}.

Passing from a discrete reconstruction problem to a continuum formulation removes many grid-dependent details and makes the underlying operators and function spaces easier to study. Once a posterior measure $\mu$ on a continuum space has been constructed, one can apply the tools of infinite-dimensional probability and functional analysis to establish well-posedness, stability, approximation, and asymptotic properties of the BIP; see, for example, \cite{stuart2010,Dashti2017,nickl2023}. 

The continuum formulation also makes it possible to study posterior contraction, consistency, information bounds, efficiency, and Bernstein--von Mises limits using methods from infinite-dimensio\-nal statistics; see \cite{knapik2011,MNP2019,MNP2021,nickl2023a} and the references therein. Continuum posterior measures can also be analyzed using information-field-theoretic and path-integral representations; see, for example, \cite{alberts2023, chang2014}. 

BIPs on infinite-dimensional state spaces have been extensively developed; see, for example, \cite{fitzpatrick1991,stuart2010,Dashti2017, nickl2023}. Most existing continuum Bayesian theories begin by specifying a prior $\nu$ and likelihood on an infinite-dimensional function space, thereby defining a posterior $\mu$, and subsequently study their finite-dimensional approximations, $\nu_N$ and $\mu_N$, respectively. This continuum-first point of view underlies the general framework of \cite{stuart2010,Dashti2017} and the approximation results in \cite{cotter2010,dashti2011}. Closely related discretization-invariance and posterior-convergence results were obtained in \cite{lassas2009, lasanen2012, lasanen2012a, helin2009, garciatrillos2018}. In these works, a continuum probability law is already part of the starting framework or is otherwise specified independently of the numerical regularization.

In this regard, it is important to note that even though a continuum unknown may be physically natural, this does not by itself provide a probability law on the corresponding function space. The existence of a continuum prior and posterior is therefore a substantive measure-theoretic assumption. 

In most computational settings, however, finite-dimensional BIPs are constructed in a different way. One begins directly with a data vector $\vec g_N$, a finite-dimensional unknown $\vec f_N$, a discrete forward operator $\CK_N$, and a matrix-valued regularizer. The resulting prior $\nu_N$ and posterior $\mu_N$ are ordinary probability measures on $\br^N$. Although a continuum model may be implicit in the motivation, no continuum prior or posterior is used.

Consequently, results proved for a postulated continuum posterior do not automatically provide a reliable guide to the behavior of the finite-dimensional posteriors arising in computation. Their numerical and theoretical relevance depends on whether the chosen discretization converges to that continuum Bayesian model. Establishing this connection is therefore a necessary step before infinite-dimensional results can be used to draw conclusions about the discrete BIPs. 

Accordingly, the direction taken in the present paper is reversed compared with the continuum-first approaches. We begin with finite-dimensional BIPs produced by a standard numerical discretization. In particular, the discrete priors are defined through local finite-difference precision matrices and need not be exact projections or finite-dimensional marginals of a pre-existing continuum Gaussian measure. We then establish that these finite-dimensional priors and posteriors \emph{determine} a continuum BIP as the discretization is refined, i.e., as $N\to\infty$. 

To clarify, in this setting the parameter $\tau$, which represents the total information contained in the data, and the regularization strength $\kappa$ are held fixed as $N\to\infty$. Thus, the effective information weight of each individual measurement is of order $\tau/N$ and tends to zero, while the total information remains fixed.


To the best of our knowledge, this is the first posterior-level discrete-to-continuum result in which both the prior $\nu_N$ and the likelihood are initially specified through a standard local numerical discretization, and the continuum posterior $\mu$ is constructed from the  finite-dimensional posteriors $\mu_N$. \emph{No measure-theoretic continuum assumption is made at the outset.}

A second setting arises when the BIP involves a small or large parameter, as in a small-noise regime. The key question is then how fine the discretization must be for the finite-dimensional BIP to approximate the relevant continuum limit, and how large the resulting error is. Mathematically, this leads to a regime in which both the discretization dimension $N$ and the total information in the data, represented by $\tau_N$, tend to infinity. The problem is therefore to identify the continuum \emph{coupled limit} of the finite-dimensional BIPs as $N\to\infty$ and $\tau_N\to\infty$ simultaneously. This limit does not follow from the fixed-$\tau$ result, since convergence as $N\to\infty$ for each fixed $\tau$ need not be uniform as $\tau\to\infty$.

There is a useful parallel between BIPs and statistical field theories. A Bayesian posterior has the form of a Gibbs measure, with the negative log-likelihood and regularization terms playing the role of an energy or action. This connection has been used in mathematical
work on path-integral formulations, information field theory, and
Langevin methods for BIPs \cite{chang2014, alberts2023, garbuno-inigo2020, liu2025}.
In this analogy, the numerical grid may be viewed as a short-distance cutoff
\cite[Sections~8.7--8.9]{zinn-justin2021}. In the related constructive field-theory program, one begins with cutoff or otherwise regularized models and proves that the regularization can be removed to obtain a well-defined continuum probability measure; see, for example, \cite{gubinelli2021}.

We use the term \emph{constructive BIP} in the same sense: the continuum prior $\nu$ and posterior $\mu$ are constructed as limits of the discrete measures. Closely related coupled continuum limits arise in statistical and Euclidean field theory. There, the lattice spacing tends to zero while the model parameters simultaneously approach a critical regime, and the resulting coupled limit defines a continuum random field; see Definition~2.1 in \cite{aizenman2020} and the discussion following it.

To summarize, the paper establishes three main results.

\paragraph{1. Constructive BIP.}
We first keep the total information parameter $\tau$ and the regularization strength $\kappa$ fixed and let $N\to\infty$. The finite-dimensional prior $\nu_N$ is the Gaussian measure generated by a local finite-difference regularizer, while the likelihood is allowed to be of a general convex generalized-linear-model (GLM) type.

We prove that the reconstructed priors $(\I_N)_\#\nu_N$ converge weakly on $L^2$ to a continuum Gaussian measure $\nu$. Here $\I_N:\br^N\to L^2$ is the reconstruction, or interpolation, operator that maps finite-dimensional vectors to continuum fields, and
$(\I_N)_\#$ denotes the pushforward of a finite-dimensional probability measure. We then analyze the discrete likelihood potentials and establish the weak convergence $(\I_N)_\#\mu_N\Rightarrow\mu$ on $L^2$ of the reconstructed posterior measures.  Thus, the continuum prior $\nu$ and posterior $\mu$ are indeed constructed as limits of the numerical models. This gives the local finite-difference regularization used in computation a direct continuum Bayesian interpretation.

Convergence results for priors of the type considered here are fairly standard. A closely related convergence theorem is proved in \cite{bolin2023}. Their general setting includes, as a special case, convergence to the same continuum Gaussian field as in our setting, but with a different discretization. We have not found the result in exactly the form required here and therefore provide a detailed proof. This is also useful because several ingredients and techniques from the proof are used elsewhere in the paper.

\paragraph{2. Coupled grid-refinement and small-noise limit.}
We next allow both the total information $\tau_N$ and the regularization strength $\kappa_N$ to tend to infinity with $N$, while keeping their ratio $\rho=\kappa_N/\tau_N$
fixed. This scaling retains both the likelihood and the regularizer in the limiting continuum cost functional. We prove that the reconstructed discrete MAP estimates converge in $L^2$ to the unique minimizer $u_*$ of the continuum functional. We then show that, in the coupled limit, the reconstructed posterior measures concentrate at $u_*$ and that their appropriately centered and rescaled fluctuations converge to a Gaussian measure.

More precisely, under explicit growth conditions relating $N$ and $\tau_N$, the rescaled centered posterior fluctuations converge to a Gaussian measure $\CN(\vec 0,Q^{-1})$ whose covariance is the inverse of $Q$, the Hessian of the continuum negative log-posterior at the MAP estimate. The proof uses the high-dimensional Laplace and Gaussian-approximation results of \cite{katsevich2026a,katsevich2025c}.


\paragraph{3. Matching of the two limits.}
Finally, we prove that the constructive and coupled limits are compatible. Starting from the continuum posterior constructed in the first part and then taking its small-noise, high-regularization limit, we recover the same continuum MAP $u_*$ and the same Gaussian fluctuation law $\CN(\vec 0,Q^{-1})$ as in the direct coupled limit. Thus the two procedures
\be
\text{discrete}\overset{N\to\infty}\longrightarrow\text{continuum}
\overset{\tau,\kappa\to\infty}\longrightarrow\text{small noise}
\ee
and
\be
\text{discrete with small noise}\overset{N,\tau_N,\kappa_N\to\infty}
\longrightarrow\text{continuum}
\ee
agree through Gaussian order. This matching validates the continuum posterior not only as a weak limit at fixed parameters but also as the correct asymptotic description of the underlying discrete BIPs in the concentrating regime.\\

Although the constructive continuum BIP may be easier to study than the discrete ones, it may still be difficult to analyze and therefore may not, by itself, provide a practically useful description of the discrete BIPs, which is our ultimate goal. The coupled limit, by contrast, yields the particularly simple approximation $\I_N\vec f_N\approx u_*+\tau_N^{-1/2}\xi$, where $\xi\sim\CN(\vec 0,Q^{-1})$; see \eqref{approx family}. Thus, over a whole
range of large values of $N$, the discrete BIPs are described in terms of $u_*$ and $Q$,
with the $N$-dependence entering only through $\tau_N^{-1/2}$. Once $u_*$ and $Q^{-1}$ have been computed, the approximation can be used for a variety of purposes, including quantifying posterior uncertainty as a function of $N$, without repeatedly analyzing the full
high-dimensional posterior. Consequently, an entire family of discrete BIPs can be studied with substantially less computational effort and, in favorable cases, even analytically.

To the best of our knowledge, this is the first result to analyze directly a coupled grid-refinement and small-noise limit for finite-dimensional Bayesian posteriors, identify their Gaussian fluctuations, and prove that these asymptotics agree with those obtained by first constructing the continuum posterior and then taking its small-noise limit.

To keep the main ideas and methodology transparent, we work in the simple one-dimensional setting of the unit circle, thereby avoiding additional complications arising from higher dimensions and boundary conditions. Extensions to more general domains and higher-dimensional settings appear feasible, but are left for future work.

It is also worth mentioning variational convergence methods, most notably $\Gamma$-convergence \cite{dalmaso1993}, which provide a natural framework for analyzing discrete regularized minimization problems. Such methods can justify convergence of minimizers and, in Bayesian terminology, of MAP estimates, but do not by themselves establish convergence of the full posterior measures.

The remainder of the paper is organized as follows. Section~\ref{sec:setting_assns} introduces the discrete inverse problem, the GLM-type likelihood, the finite-difference
regularization, and the standing assumptions. Section~\ref{sec:gen estims} develops general estimates for the discrete Hessians and the corresponding whitened functionals.
Section~\ref{sec:constr BIP} constructs the continuum prior and posterior at fixed $\tau$ and $\kappa$ and proves weak convergence of the reconstructed finite-dimensional posterior measures. Section~\ref{sec:small-noise-high-reg} studies the coupled grid-refinement and small-noise limit, including convergence of the MAP estimates, posterior concentration, and Gaussian fluctuations. Section~\ref{sec:Compatibility} proves that the constructive and coupled limits agree through Gaussian order. The appendices prove auxiliary results needed in the main text.

\section{Inverse problem setting and assumptions}\label{sec:setting_assns}

Let $S^1$ denote the unit circle in $\br^2$. We identify $S^1$ with the interval $[0,1]$ by gluing the endpoints. Also, we sometimes view functions on $S^1$ as 1-periodic functions on $\br$. Throughout the paper we use the notation $C:=C(S^1)$ and $L^p:=L^p(S^1)$ for any $p\in[1,\infty]$. The notation $a_N\les b_N$ means that there exists some $c>0$ independent of $N$ such that $a_N\le cb_N$ for all $N\in\N$. Also, for a vector $\vec f\in\br^N$, we denote $\|\vec f\|^2_{\ell^2}=\sum_{i=0}^{N-1} f_i^2$. Unless stated otherwise, all limits are taken as $N\to\infty$. The symbol $\Rightarrow$ denotes weak convergence of probability measures.

Consider the following equation  
\be\label{problem}
(\CK u)(t)=g(t),\quad t\in S^1.
\ee
We assume that $\CK:L^2\to C$ is a bounded linear operator.  

Assume that we have $N$ measurements, which represent local averages of some function $g\in C$ representing the background (unknown) continuum data:
\bs\label{BIP data}
g_i=&N\int_{S^1} w(N(t-t_i))g(t)\dd t,\ t_i=i/N,\ i=0,1,\dots, N-1,\\
\vec g_N:=&(g_0,\dots,g_{N-1})\in\br^N.
\es
Here $w(t)$ is the detector aperture function. 


For each $N$, the posterior is understood conditionally on the observed data vector $\vec g_N$. This is the standard viewpoint in BIPs: once the data have been observed, they are treated as fixed when the posterior distribution of the unknown is analyzed; see, for example, \cite{stuart2010, cotter2010, Dashti2017}. To obtain a consistent family of discrete BIPs as $N\to\infty$, we assume that the data vectors $\vec g_N$ are generated by \eqref{BIP data} from the same fixed background profile $g\in C$. Thus, increasing $N$ corresponds to observing the same underlying data profile at progressively finer resolution, rather than changing the underlying data from one problem to the next.

We assume that the measurements contain independent errors and allow their conditional distributions to have a scaled exponential-family form of the type underlying many generalized linear models (GLMs). Thus, an observation $v$ associated with a predicted value $r$ at location $t$ has conditional density 
\be\label{glm pointwise density}
p_q(v\mid r,t)
=\exp\left\{q[v\vartheta-B(\vartheta)]+c(v,t,q)\right\},\quad \vartheta=\vartheta(r,t),
\ee
where $\vartheta$, $B$, and $c$ are some functions; see, for example, \cite[eq. (2.4)]{mccullagh2019}. Note that the function $c$ is independent of $r$. The parameter $q>0$ determines the information scale of the measurement. Depending on the observation model, $q$ may represent a precision, inverse-dispersion, exposure, or effective sample-size parameter. This class includes Gaussian models with varying precision, gamma and other exponential-dispersion models, and Poisson models with varying exposure, among other standard observation models. Up to an additive term independent of $r$, the corresponding pointwise negative log-likelihood is $q\Phi(r,v,t)$, where
\be
\Phi(r,v,t)=B\bigl(\vartheta(r,t)\bigr)-v\vartheta(r,t).
\ee

This representation motivates our likelihood model. However, the analysis below permits a more general pointwise negative log-likelihood $\Phi(r,v,t)$ that need not arise from an exponential family. We retain, however, the information-scaling property that, up to an additive term independent of $r$, the negative log-likelihood of an observation with information parameter $q$ is $q\Phi(r,v,t)$. Following the above convention, we continue to assume that $r$ is the predicted value, $v$ is the observed value, and $t\in S^1$ is the measurement location. Thus, the observations need not be identically distributed, and $\Phi$ may depend explicitly on $t$.

Define the reconstruction (or interpolation) map in the image domain:
\be\label{recon map}
\I_N:\br^N\to L^2,\quad (\I_N \vec f)(t)=f_m,\ t\in[m/N,(m+1)/N),\ 0\le m\le N-1.
\ee
We also introduce the reconstruction map in the data domain:
\be\label{data recon map}
\I_N^d:\br^N\to L^2,\quad (\I_N^d \vec g)(t)=\sum_{i=0}^{N-1} w(N(t-t_i))g_i,
\ee
Note that the operator in \eqref{BIP data} is $(N\I_N^d)^*$. 

Throughout, $\I_N^*$ and $(\I_N^d)^*$ denote adjoints with respect to the standard Euclidean pairing on $\br^N$ and the $L^2$ pairing on the continuum space. 

\begin{assumption}\label{ass:det apert}$ $
The detector aperture function $w$ satisfies
\begin{enumerate}
\item\label{w Linf} $w\in L^\infty(\br)$ and $w$ is compactly supported. 
\item\label{w norm} $w$ is normalized: $\int_{\br}w(t)\dd t=1$. 
\end{enumerate}
\end{assumption}

\noindent
Note that the function $w$ is not periodized. Whenever $w$ is integrated against a function on $S^1$, the latter is understood through its $1$-periodic extension to $\br$, so that the aperture wraps around the endpoints of the representative interval $[0,1]$.

We discretize the integral in \eqref{problem} into a sum with $N$ terms. The negative log-posterior functional $\Psi_N^{\mathrm{full}}$ and the likelihood functional $\Psi_N$ are given by
\bs\label{discr fnls}
\Psi_N^{\mathrm{full}}(\vec f\,):=&\Psi_N(\vec f\,)+\frac{\kappa}{2N}\sum_{i=0}^{N-1}\big[f_i^2+(D_N\vec f\,)_i^2\big],\\
\Psi_N(\vec f\,):=&\frac{\tau}N\sum_{i=0}^{N-1}\Phi\big((\CK_N \vec f\,)_i,g_i,t_i\big),
\es
where $\kappa>0$ is the regularization parameter and
\bs\label{KN def}
\CK_N =& (N\I_N^d)^*\CK\I_N:\br^N\to\br^N,\\
(D_N\vec v)_i:=&N(v_{i+1}-v_i), \quad 0\le i\le N-1.
\es
Informally, if $K(t,s)$ is the kernel of $\CK$, then 
\bs\label{using kernel}
(\CK_N)_{i,m}=&N\int_{S^1}w\Big(\frac{t-t_i}{1/N}\Big)\int_{m/N}^{(m+1)/N}  K(t,s)\dd s \dd t, \quad 0\le i,m\le N-1.
\es
However, we will never use this representation in the paper.

The dependence of $\Psi_N^{\mathrm{full}}$ and $\Psi_N$ on data is omitted from notation for simplicity. We will also refer to $\Psi_N^{\mathrm{full}}$ as the cost functional. In \eqref{discr fnls} and below, we use the periodic convention $f_N=f_0$.  Assumption~\ref{ass:det apert}\eqref{w norm} and \eqref{data recon map} imply that $\vec g_N=(N\I_N^d)^*g$ is the vector of local averages of $g$.


The data vector $\vec g_N$, discrete unknown $\vec f$, discrete forward operator $\CK_N$, likelihood functional, and finite-dimensional regularizer defined above constitute a
conventional computational formulation of a finite-dimensional BIP. In particular, \eqref{discr fnls} and \eqref{KN def} are based on a standard numerical discretization of integral equations of the form \eqref{problem}.

We interpret $\tau$ as a proxy for the total information contained in all $N$ measurements. According to \eqref{discr fnls}, the information contributed by each individual measurement scales as $\tau/N$ and therefore decreases as $N$ grows. Setting $q_N=\tau/N$ in \eqref{glm pointwise density} and ignoring an additive constant independent of $\vec f$, we get the negative log-likelihood in \eqref{discr fnls}:
\be
-\sum_{i=0}^{N-1}
\log p_{q_N}\bigl(g_i\mid(\CK_N\vec f)_i,t_i\bigr)
=
\frac{\tau}{N}
\sum_{i=0}^{N-1}
\Phi\bigl((\CK_N\vec f)_i,g_i,t_i\bigr).
\ee

Let $\vec u_{*,N}=\mathrm{argmin}\Psi_N^{\mathrm{full}}(\vec f\,)$ be the MAP estimate associated with the observed data $\vec g_N$. Below we assume that $\Phi$ is convex and continuous in its first argument and bounded from below. Together with the positive definite quadratic regularization term, this implies that $\Psi_N^{\mathrm{full}}$ is coercive and strongly convex. Consequently, $\vec u_{*,N}$ exists and is unique.

Then the first variation of $\Psi_N^{\mathrm{full}}$ at $\vec u_{*,N}$ vanishes.  Equivalently, for every $\theta\in\br^N$,
\be\label{inv EL weak}
\frac{\tau}N\sum_{i=0}^{N-1}\alpha_i(\CK_{N}\vec \theta)_i
+
\frac{\kappa}{N}\sum_{i=0}^{N-1}\left[(\vec u_{*,N})_i \theta_i+
(D_N \vec u_{*,N})_i(D_N\vec \theta)_i\right]=0.
\ee
Here and below we use the following notation
\bs\label{inv alpha c}
\alpha_i:=&\pa_1\Phi(r_i,g_i,t_i),\quad
b_i:=\pa_1^2\Phi(r_i,g_i,t_i),\quad r_i:=(\CK_{N}\vec u_{*,N})_i,\quad 0\le i\le N-1,
\es
and $\pa_1$ denotes differentiation of $\Phi$ with respect to its first argument.

\begin{assumption}\label{ass:setting}$ $
\begin{enumerate}
\item\label{Kbdd} The operators $\CK$ and $\CK^*$ are bounded from $L^2$
to $C$.
\item\label{bmin} One has $\Phi(p,g(t)+h,t)\gtrsim -1$ and $\pa_1^2\Phi(p,g(t)+h,t)\ge 0$ for all $p\in\br$, $t\in S^1$, and $|h|\le h_0$ for some $h_0>0$.
\item\label{Phi lip} The function $\Phi$ is locally Lipschitz in $(p,v,t)$, and
$\partial_1^l\Phi$, $l=1,2$, are continuous in $(p,v,t)$.
\item\label{higher ders bnd} For some $p_0>0$, 
\bs\label{data higher der bnd inv}
&\sup_{N\in\N}\sup_{0\le i\le N-1}\ 
\sup_{|p|\le p_0}
\big|\pa_1^3\Phi(r_i+p,g_i,t_i)\big|\les 1.
\es
\item\label{common data} The background data profile satisfies $g\in C$.
\end{enumerate}
\end{assumption}

\noindent
To clarify, all the absorbed constants, $p_0$, and $h_0$ are independent of $N$. 

Unless stated otherwise, all results in the remainder of the paper are proved under Assumptions~\ref{ass:det apert} and~\ref{ass:setting}. These assumptions will therefore not be repeated in the statements of individual theorems and lemmas.

When we consider Gaussian fluctuations in the coupled limit in Section~\ref{ssec:Gaus scale}, see Lemmas~\ref{lem:conv of gaus meas} and \ref{lem:bip final gaussian fluct}, 
we impose an additional condition.

\begin{assumption}\label{ass:extra}$ $
Suppose at least one of the two conditions below holds:
\begin{enumerate}
\item\label{apert extra} $\sum_{i\in\BZ}w(t-i)=1$ for a.e. $t\in\br$.
\item\label{K extra} $\CK^*:L^2\to C$ is compact. 
\end{enumerate}
\end{assumption}

\section{General estimates}\label{sec:gen estims}
Denote the value of the cost functional $\Psinf$ at the MAP estimate $\vec u_{*,N}$ by
\be\label{inv classical action}
S_{N}:=
\frac{\tau}N\sum_{i=0}^{N-1}\Phi(r_i,g_i,t_i)
+\frac{\kappa}{2N}\sum_{i=0}^{N-1}\big[(\vec u_{*,N})_i^2+(D_N \vec u_{*,N})_i^2\big].
\ee
Next define the cubic-and-higher remainders
\bs\label{inv nonlinear remainders}
\Phi_{i,N}(p):=&
\Phi(r_i+p,g_i,t_i)-\Phi(r_i,g_i,t_i)
-\alpha_i p-\tfrac12 b_i p^2.
\es
Thus
\bs\label{inv remainders vanish}
\Phi_{i,N}(0)=&\Phi_{i,N}'(0)=\Phi_{i,N}''(0)=0.
\es

Writing $\vec f_N=\vec u_{*,N}+\vec \eta$, the cost functional becomes
\bs\label{Psi inv parallel}
\Psinf(\vec f_N)=&
S_{N}
+\frac12\langle \vec \eta,Q_N\vec \eta\rangle_{\ell^2}+\frac{\tau}N\sum_{i=0}^{N-1}
\Phi_{i,N}\big((\CK_{N}\vec \eta)_i\big),
\quad \vec \eta\in\br^N,
\es
where $Q_N:\br^N\to\br^N$ is the self-adjoint Hessian at $\vec u_{*,N}$, defined by
\bs\label{Q inv parallel}
\langle \vec \eta,Q_N\vec \eta\rangle_{\ell^2}
=&
\frac{\tau}N\sum_{i=0}^{N-1}
b_i(\CK_{N}\vec \eta)_i^2
+\frac{\kappa}{N}\sum_{i=0}^{N-1}\big[
\eta_i^2+(D_N\vec \eta)_i^2\big].
\es
The linear terms cancel because of \eqref{inv EL weak}.  

Define $A_N=Q_N^{-1/2}:\br^N\to\br^N$. Let $\mu_{\min}(Q_N)$ denote the smallest eigenvalue of $Q_N$.

\begin{lemma}\label{lem:QNAN BIP}
One has
\be\label{eq:A-bounds}
\mu_{\min}(Q_N)\gtrsim N^{-1},
\quad
\norm{A_N}_{\ell^2\to\ell^\infty}\les 1.
\ee
\end{lemma}

\begin{proof}
Introduce the normalized discrete norm
\bs\label{norm norm}
\|\vec v\|_N^2:=&\frac1N\sum_{i=0}^{N-1}v_i^2,
\quad \vec v\in\br^N .
\es
The first bound in \eqref{eq:A-bounds} is immediate.

By the compact support and normalization of $w$ (Assumption~\ref{ass:det apert}), the uniform continuity of $g$ (Assumption~\ref{ass:setting}\eqref{common data}), and the definition of $g_i$ in \eqref{BIP data}, $\max_{0\leq i\leq N-1}|g_i-g(t_i)|\to0$. Hence, for all sufficiently large $N$, we may write $g_i=g(t_i)+h_{i,N}$ with $|h_{i,N}|\leq h_0$. Therefore, \eqref{inv alpha c}, Assumption~\ref{ass:setting}\eqref{bmin}, and \eqref{Q inv parallel} imply that $b_i\ge0$ and 
\be\label{QN controls R inv}
\norm{D_N\vec \eta}_N^2+\norm{\vec \eta}_N^2
\les \langle \vec \eta,Q_N\vec \eta\rangle_{\ell^2}.
\ee
Therefore $\mu_{\min}(Q_N)\gtrsim N^{-1}$. 

Further, the standard one-dimensional discrete Sobolev estimate (see, e.g., \cite{nagai2009, yamagishi2012} for more general discussions) gives
\be\label{disc Sobolev inv}
\norm{\vec \eta}_{\ell^\infty}^2
\les
\norm{D_N\vec \eta}_N^2+\norm{\vec \eta}_N^2 .
\ee
Combining \eqref{QN controls R inv} and \eqref{disc Sobolev inv}, we obtain
$\norm{\vec \eta}_{\ell^\infty}^2 \les \langle \vec \eta,Q_N\vec \eta\rangle_{\ell^2}$. 
Now put $\vec \eta=A_N\vec z=Q_N^{-1/2}\vec z$.  Since $Q_N$ is symmetric positive definite, we prove the second bound:
\be
\norm{A_N\vec z}_{\ell^\infty}^2
\les
\langle A_N\vec z,Q_NA_N\vec z\rangle_{\ell^2}
=\|\vec z\|_{\ell^2}^2.
\ee 
\end{proof}

Define the whitened functional
\be\label{inv whitened exponent}
\psi_N(\vec x):=\Psinf(\vec u_{*,N}+A_N\vec x)-S_{N},\quad \vec x\in\br^N.
\ee
By construction,
\be\label{at minimizer}
\psi_N(0)=0,\quad D\psi_N(0)=0,\quad D^2\psi_N(0)=I_{N},
\ee
where $I_N$ is the $N\times N$ identity matrix.

Recall that the operator norm of a symmetric $k$-th order tensor $A_k$ is given by \cite{zhang2012}
\be\label{T norm}
\Vert A_k\Vert_{\op}:=\sup_{\|u\|_{\ell^2}=1} \big|A_k[u^{\otimes k}]\big|.
\ee
For a function $f\in C^k(\br^N)$, the $k$-th derivative tensor is
\be
(\nabla^kf(\vec x))_{j_1\dots j_k} = \pa_{x_{j_1}}\dots\pa_{x_{j_k}}f(\vec x),\quad
1\le j_1,\dots j_k\le N,\quad \vec x\in\br^N. 
\ee

\begin{lemma}\label{lem:psi ders bdd} There exists $r_0>0$ such that
\be\label{inv whitnd ders bnd}
\sup_{\|\vec x\|_{\ell^2}\le r_0}
\big\|D^3\psi_N(\vec x)\big\|_{\op} \les 1.
\ee
\end{lemma}

Thus the third derivative tensor of $\psi_N$ is uniformly bounded.

\begin{proof}
Let $\vec \eta=A_N\vec x$. Then, for $l=3$,
\bs\label{inv Dl exact}
\pa_{x_{j_1}}\dots\pa_{x_{j_l}}\psi_N(\vec x)
=&
\frac{\tau}N\sum_{i=0}^{N-1}
\pa_1^l\Phi\left(r_i+(\CK_{N}A_N\vec x)_i,\, g_i,\, t_i\right)
\prod_{r=1}^l (\CK_{N}A_N)_{i,j_r}.
\es
This implies
\bs\label{case r0 BIP}
D^l \psi_N(\vec x)[\vec z^{\otimes l}]=&\sum_{j_1,\dots,j_l}\big[\pa_{x_{j_1}}\dots \pa_{x_{j_l}}\psi_N(\vec x)\big]z_{j_1}\dots z_{j_l}\\
=& \frac{\tau}N\sum_{i}\pa_1^l\Phi(\cdot)
\Big[\sum_{j_1,\dots,j_l}(\CK_{N}A_N)_{i,j_1}\dots (\CK_{N}A_N)_{i,j_l}z_{j_1}\dots z_{j_l}\Big]\\
=&\frac{\tau}N\sum_{i}\pa_1^l\Phi(\cdot)\mu_i^l,\quad \vec \mu:=\CK_{N} A_N \vec z.
\es
Here $\vec z\in\br^N$ is a tangent vector to $\br^N$ at the point $\vec x$.

By Lemma~\ref{lem:QNAN BIP}, $\norm{A_N}_{\ell^2\to\ell^\infty}\les 1$. Therefore, by Assumption~\ref{ass:det apert}\eqref{w Linf}, for every $\vec v\in\br^N$,
\bs\label{CKAN linf}
\|\CK_{N}A_N\vec v\|_{\ell^\infty} \les \|\CK \I_N (A_N\vec v)\|_{C}
\les \|\I_N (A_N\vec v)\|_{L^2}\les \|A_N\vec v\|_N \le \|A_N\vec v\|_{\ell^\infty}\les \|\vec v\|_{\ell^2}.
\es
Moreover, if $\|\vec x\|_{\ell^2}\le r_0$ and $r_0>0$ is sufficiently small, then $\|\CK_{N}A_N\vec x\|_{\ell^\infty}\le p_0$ uniformly in $N$. Hence the derivative bound \eqref{data higher der bnd inv} applies in \eqref{case r0 BIP}.

Now suppose $\|\vec z\|_{\ell^2}\le1$. Using the derivative bound \eqref{data higher der bnd inv} and \eqref{CKAN linf} in \eqref{case r0 BIP}, we obtain
\bs
\big|D^l\psi_N(\vec x)[\vec z^{\otimes l}]\big|
\les & \frac1{N}\sum_{i=0}^{N-1}
\big|(\CK_{N}A_N\vec z)_i^l\big|
\les \frac1N\sum_{i=0}^{N-1}1 \les 1.
\es
Taking the supremum over $\|\vec z\|_{\ell^2}\le1$ proves \eqref{inv whitnd ders bnd}.
\end{proof}

\section{Constructive BIP at fixed information and \protect\\ regularization strengths}\label{sec:constr BIP}

The functional $\Psinf$ in \eqref{discr fnls} determines the posterior measure $\mu_N(\dd\vec f)\propto \exp\{-\Psinf(\vec f)\}\dd\vec f$ on $\br^N$. We embed this posterior into the continuum state space $L^2$ using the map $\I_N:\br^N\to L^2$; see \eqref{recon map}. Our goal is to prove that the pushforwards of $\mu_N$ converge to a continuum posterior.

\begin{theorem}\label{thm:poster conv} There exists a probability measure $\mu$ on $L^2$ such that $(\I_N)_\#\mu_N\Rightarrow \mu$ in $L^2$.
\end{theorem}

The theorem is proved in Sections~\ref{ssec:conv of priors}--\ref{ssec:conv posters}. In Section~\ref{ssec:relation-discretization-invariance}, we compare our result with the most closely related existing approaches.

\subsection{Convergence of priors}\label{ssec:conv of priors}
By \eqref{discr fnls}, we consider the following quadratic regularizer
\bs\label{LN prep}
\frac{\kappa}2\langle \vec f,L_N \vec f\rangle_{\ell^2}
= \frac{\kappa}{2N} \sum_{i=0}^{N-1}\big[f_i^2+(D_N \vec f)_i^2\big],
\quad \vec f\in\br^N, 
\es
so
\be\label{LN def}
(L_N \vec f)_i=\tfrac1N[f_i-N^2(f_{i+1}-2f_i+f_{i-1})],\quad 0\le i\le N-1.
\ee

Define the finite-dimensional and continuum Gaussian priors 
\bs\label{two priors}
\nu_N=\CN(\vec 0,(\kappa L_N)^{-1}),\quad
\nu=\CN(\vec 0,(\kappa\CL)^{-1}),\ \CL=-\pa_t^2+\Id,
\es
where $\Id$ is the identity operator.

\begin{lemma}\label{lem:conv of priors}
The interpolated priors converge weakly,
$\widetilde\nu_N:=(\I_N)_\#\nu_N \Rightarrow \nu$ on $L^2$.
\end{lemma}

The lemma is proved in Appendix~\ref{sec:conv prior}.

\subsection{Convergence of likelihood potentials}\label{ssec:conv lklhd}
Define the continuum likelihood similarly to the discrete one in \eqref{discr fnls}
\bs\label{Psi def}
\Psi(f)=\tau\int_{S^1}\Phi\bigl((\CK f)(t),g(t),t\bigr)\dd t,\quad f\in L^2.
\es
Recall that $\I_N^d$ is the reconstruction map in the data domain defined in \eqref{data recon map}. Pick any $f\in L^2$. 

\begin{lemma}\label{lem:PsiNPsi} Let $\vec f_N\in\br^N$ be any sequence such that $\I_N\vec f_N\to f$ in $L^2$. One has $\Psi_N(\vec f_N)\to \Psi(f)$.
\end{lemma}

Clearly, one such choice is $\vec f_N=(N\I_N)^*f$. In this case convergence in $L^2$ follows by noting that $P_N:=\I_N(N\I_N)^*$ is the $L^2$-orthogonal projection onto the space of functions constant on the intervals $(i/N,(i+1)/N)$, $i=0,1,\dots, N-1$.

\begin{proof}
By Assumption~\ref{ass:setting}\eqref{Kbdd}, $\CK:L^2\to C$ is continuous. Therefore, $\CK(\I_N \vec f_N) \to \CK f$ in $C$. By the properties of $w$, $\CK_N\vec f_N$ and $\vec g_N$ are the vectors of local averages of $\CK (\I_N \vec f_N)$ and $g$, respectively. By Assumption~\ref{ass:setting}\eqref{Phi lip}, $\Phi$ is Lipschitz continuous on the relevant compact set. Therefore
\bs\label{PsiNPhi}
\Psi_N(\vec f_N)=\frac{\tau}N\sum_{i=0}^{N-1} \Phi\bigl((\CK_N \vec f_N)_i,g_i,t_i\bigr)\to&\,\tau\int_{S^1}\Phi\bigl((\CK f)(t),g(t),t\bigr)\dd t=\Psi(f).
\es
\end{proof}

We now prove another result that is used in the following section. Pick any compact set $K\subset L^2$. By the continuity of $\CK$, the set $\CK(K)\subset C$ is compact. Hence  
\be\label{proj C error}
\sup_{h\in \CK(K)}\|h-\I_N(N\I_N^d)^* h\|_{L^\infty} \to0.
\ee
Pick any sequence $\I_N \vec f_N\in K$. By \eqref{Psi def} we have
\bs\label{Psi vecf}
\Psi(\I_N \vec f_N)
=&\tau\int_{S^1}\Phi\bigl((\CK \I_N\vec f_N)(t),g(t),t\bigr)\dd t.
\es
Using \eqref{proj C error}, where $h$ has the meaning of $h=\CK\I_N \vec f_N$,
the same argument as in the proof of Lemma~\ref{lem:PsiNPsi} now gives
\be\label{Psi conv bnd}
\sup_{\vec f_N:\, \I_N \vec f_N\in K}\big|\Psi_N(\vec f_N)-\Psi(\I_N \vec f_N)\big|\to0.
\ee

%

\subsection{Convergence of normalizing constants}
Recall that the priors $\nu_N$ and $\nu$ are defined in \eqref{two priors}. Define the normalizing constants
\bs\label{two Zs}
Z_N=&\int_{\br^N}\exp\{-\Psi_N(\vec f_N)\}\dd\nu_N(\vec f_N),\quad
Z=\int_{L^2}\exp\{-\Psi(f)\}\dd\nu(f).
\es

\begin{lemma}\label{lem:ZNtoZ} One has $Z_N\to Z$ and $0<Z<\infty$.
\end{lemma}

By the definition of push-forward:
\be\label{Psiaux}
\int_{L^2}\exp\{-\Psi(f)\}\dd(\I_N)_\#\nu_N(f)
=\int_{\br^N}\exp\{-\Psi(\I_N\vec f_N)\}\dd \nu_N(\vec f_N).
\ee
By Assumption~\ref{ass:setting}\eqref{bmin}, 
\bs\label{unif int}
&\Psi_N(\vec f_N)\ge -c,\quad 0<\exp\{-\Psi_N(\vec f_N)\}\le e^c,\quad \vec f_N\in\br^N,\\ 
&\Psi(f)\ge -c,\quad 0< \exp\{-\Psi(f)\}\le e^c,\quad f\in L^2,
\es
for some $c>0$. Since $f\to\exp\{-\Psi(f)\}$ is a bounded, positive, continuous function on $L^2$, we immediately get that $0<Z<\infty$. Furthermore, Lemma~\ref{lem:conv of priors} and \eqref{Psiaux} imply
\bs\label{Zconv I}
\int_{\br^N}\exp\{-\Psi(\I_N \vec f_N)\}\dd\nu_N(\vec f_N)\to 
\int_{L^2}\exp\{-\Psi(f)\}\dd\nu(f).
\es
Lemma~\ref{lem:ZNtoZ} is proved once we establish the following result.  

\begin{lemma}
One has
\be\label{Zconv II}
\int_{\br^N}
\big|\exp\{-\Psi_N(\vec f_N)\}
-\exp\{-\Psi(\I_N \vec f_N)\}\big|
\dd\nu_N(\vec f_N) \to 0.
\ee
\end{lemma}

\begin{proof} 
Fix $\varepsilon>0$. In Appendix~\ref{ssec:Gaussian tightness} we proved that the family $\widetilde\nu_N=(\I_N)_\#\nu_N$ is tight. Hence there exists a compact set $K_\varepsilon\subset L^2$ such that, for all sufficiently large $N$,
$\widetilde\nu_N(K_\varepsilon)\ge 1-\varepsilon$. Equivalently, if
\be
A_{N,\varepsilon}:=\{\vec f_N\in\br^N:\ \I_N \vec f_N\in K_\varepsilon\},
\ee
then $\nu_N(A_{N,\varepsilon})\ge 1-\varepsilon$.

We split the desired integral into its contribution over $A_{N,\varepsilon}$ and over
$A_{N,\varepsilon}^c$. On $A_{N,\varepsilon}$, \eqref{Psi conv bnd} gives
\be
\delta_{N,\varepsilon}:=
\sup_{\vec f_N\in A_{N,\varepsilon}}
\big|\Psi_N(\vec f_N)-\Psi(\I_N\vec f_N)\big|\to 0.
\ee
The map $x\mapsto e^{-x}$ is Lipschitz on $[-a,\infty)$ with Lipschitz constant
$e^a$ for any $a\in\br$. Since both $\Psi_N$ and $\Psi$ are bounded below by $-c$ (see \eqref{unif int}), we have
\bs\label{in Ae}
&\int_{A_{N,\varepsilon}}
\left|
e^{-\Psi_N(\vec f_N)}
-e^{-\Psi(\I_N \vec f_N)}
\right|\,\dd\nu_N(\vec f_N)  \\
&\quad \le
e^c\sup_{\vec f_N\in A_{N,\varepsilon}}\left|
\Psi_N(\vec f_N)-\Psi(\I_N\vec f_N)\right|
= e^c\delta_{N,\varepsilon} \to 0.
\es
On the complement $A_{N,\varepsilon}^c$, we can use the right bounds in \eqref{unif int}.
Hence
\bs\label{in Aec}
&\int_{A_{N,\varepsilon}^c}
\left|
e^{-\Psi_N(\vec f_N)}-
e^{-\Psi(\I_N\vec f_N)}\right|
\,\dd\nu_N(\vec f_N)  
\le 2e^c\nu_N(A_{N,\varepsilon}^c)
\le 2e^c\varepsilon.
\es
Combining \eqref{in Ae} and \eqref{in Aec} yields
\be
\limsup_{N\to\infty}
\int_{\br^N}
\left|e^{-\Psi_N(\vec f_N)}-e^{-\Psi(\I_N \vec f_N)}\right|\,d\nu_N(\vec f_N)
\le 2e^c\varepsilon.
\ee
Since $\varepsilon>0$ is arbitrary, the claim follows.
\end{proof}

\subsection{Convergence of posteriors}\label{ssec:conv posters}
Define the finite-dimensional and continuum posteriors by
\be\label{two posters}
\frac{\dd\mu_N}{\dd\nu_N}(\vec f_N)
=\frac1{Z_N}\exp\{-\Psi_N(\vec f_N)\},\quad
\frac{\dd\mu}{\dd\nu}(f)
=\frac1{Z}\exp\{-\Psi(f)\}.
\ee
Note that $\mu_N$ is precisely the posterior measure introduced at the beginning of Section~\ref{sec:constr BIP} and used in Theorem~\ref{thm:poster conv}. Then for every bounded continuous $F:L^2\to\br$,
\be\label{last conv}
\begin{aligned}
\int_{L^2} F(u)\dd(\I_N)_\#\mu_N(u)
&=\frac{\int_{\br^N}F(\I_N\vec f_N)
\exp\{-\Psi_N(\vec f_N)\}\dd\nu_N(\vec f_N)}{Z_N}  \\
&\to \frac{\int_{L^2} F(f)\exp\{-\Psi(f)\}\dd\nu(f)}{Z}  \\
&=\int_{L^2} F(f)\dd\mu(f).
\end{aligned}
\ee

To prove convergence in the last equation, we only need to establish convergence of the numerators. This is done analogously to the preceding section using the stated properties of $F$. Indeed, the analog of \eqref{Zconv I} becomes
\bs\label{Fconv I}
\int_{\br^N}F(\I_N \vec f_N)\exp\{-\Psi(\I_N \vec f_N)\}\dd\nu_N(\vec f_N)\to 
\int_{L^2}F(f)\exp\{-\Psi(f)\}\dd\nu(f).
\es
The analog of \eqref{Zconv II} becomes
\be\label{Fconv II}
\int_{\br^N}
\big|F(\I_N \vec f_N)\exp\{-\Psi_N(\vec f_N)\}-F(\I_N \vec f_N)\exp\{-\Psi(\I_N \vec f_N)\}\big|
\dd\nu_N(\vec f_N) \to 0.
\ee
Since $F(\I_N\vec f_N)$ factors out of the difference inside the absolute value and $F$ is bounded, \eqref{Zconv II} implies \eqref{Fconv II}. Thus \eqref{last conv} and, therefore, Theorem~\ref{thm:poster conv} is proved.

\subsection{Relation to existing continuum BIP approximation frameworks}
\label{ssec:relation-discretization-invariance}

We now discuss how the construction in this section differs from the existing approaches to BIPs developed in \cite{lassas2009}, \cite{stuart2010, cotter2010, dashti2011, Dashti2017}, and \cite{lasanen2012, lasanen2012a}. For the purposes of the comparison with the present paper, the relevant common feature of these works is that the infinite-dimensional BIP is part of the starting framework. One specifies a continuum unknown, a continuum prior law, and a continuum observation likelihood, and then studies finite-dimensional approximations of this already formulated continuum BIP.

In the paper we start from the finite-dimensional Bayesian model \eqref{discr fnls}, \eqref{KN def}, with local finite-difference Gaussian priors and discrete GLM-type likelihoods, \emph{as they arise in numerical regularization}, see \eqref{LN prep}--\eqref{two priors}. The continuum prior $\nu$ and posterior $\mu$ are then recovered as weak limits of the reconstructed finite-dimensional priors $\nu_N$ and posteriors $\mu_N$, respectively; see Lemma~\ref{lem:conv of priors} and Theorem~\ref{thm:poster conv}. 

To better understand this difference, let us recall what happens when a continuum-first construction is discretized at the level of the prior. A common starting point is a continuum Gaussian prior $\nu=\CN(\vec 0,L^{-1})$, where $L$ is the continuum elliptic operator. Other priors are possible too, such as Besov priors \cite{lassas2009}. Then one defines a finite-dimensional prior by projection or truncation. Let $T_N$ denote the corresponding operator. The projected prior has the form $\nu_N^{\rm proj}:=(T_N)_\#\nu$ and the covariance is thus $T_N L^{-1}T_N^*$. 

In the paper, we start with the discrete quadratic regularizer $\tfrac12\langle \vec f,L_N\vec f\rangle_{\ell^2}$, so the finite-dimensio\-nal prior is $\nu_N = \CN(\vec 0,L_N^{-1})$. Here $L_N$ is a finite-difference approximation of $L$. For simplicity, in this comparison we absorbed the regularization parameter $\kappa$ into the corresponding operators.
In general, $L_N^{-1} \neq T_N L^{-1} T_N^*$. 

This distinction is important. In the continuum-first approach, the corresponding precision operator $(T_N L^{-1}T_N^*)^{-1}$ need not have a simple finite-difference, banded, or sparse structure. By contrast, the precision $L_N$ in the present construction is local. This is the form \emph{naturally used in numerical inverse problems}. Our approach therefore justifies the natural choice of local finite-difference regularizers, which are actually used in computations. We show that such priors have the desired continuum Bayesian interpretation.

A separate distinction concerns the likelihood. In \cite{lassas2009}, the observation model is linear with additive Gaussian white noise. In \cite{stuart2010, cotter2010, dashti2011, Dashti2017}, the forward map may be nonlinear, but the standard observation model is additive Gaussian noise, leading to a quadratic data-misfit in the residual. The present construction treats integrated GLM-type likelihoods directly at the discrete-to-continuum level. Thus the likelihood need not be quadratic in the predicted data variable and need not arise from additive Gaussian noise.

A more general approach is developed in \cite{lasanen2012a}. The paper provides an abstract framework for posterior convergence when an infinite-dimensional unknown is replaced by finite-dimensional, or more generally approximating, unknowns. Here again a continuum prior is assumed to exist. 

In \cite{lasanen2012a} the discrete priors do not have to be exact projections of the continuum prior. From this point of view, one might try to place the present construction inside the framework of \cite{lasanen2012a} by taking the approximating unknown to be the reconstructed function $\I_N\vec f_N$. However, doing so would not make the present result automatic. The main model-specific work would then be precisely to verify the hypotheses established in this section: the convergence of the local finite-difference priors and the consistency of the discrete likelihoods established in Sections~\ref{ssec:conv lklhd}--\ref{ssec:conv posters}.

\section{Small-noise, high-regularization scaling for GLM-type \protect\\ likelihoods}
\label{sec:small-noise-high-reg}

\subsection{Problem setting. Coupled limit}\label{ssec:diag lim setting}
We now specify the scaling regime in which the BIP posterior becomes increasingly concentrated. We assume that the total amount of information $\tau$ in \eqref{discr fnls} \emph{increases} as $N$ grows. In other words, we replace $\tau$ with $\tau_N$ in \eqref{discr fnls} and let $\tau_N\to\infty$. To make sure the likelihood and regularization terms remain balanced, we also replace $\kappa$ with $\kappa_N$ in \eqref{discr fnls} and let $\kappa_N\to\infty$. We assume that the ratio $\kappa_N/\tau_N$ remains constant.

Omitting an additive constant independent of $\vec f$, the negative log-posterior is therefore $\tau_N \Psinf(\vec f)$, where
\bs\label{PsiN def}
\Psinf(\vec f_N):=&\Psi_N(\vec f)+\frac{\rho}{2N}\sum_{i=0}^{N-1}\big[f_i^2+(D_N\vec f_N)_i^2\big],\\ 
\Psi_N(\vec f):=&\frac1N\sum_{i=0}^{N-1}\Phi\big((\CK_{N}\vec f_N)_i,g_i,t_i\big),\quad \rho=\frac{\kappa_N}{\tau_N},
\es
and $\rho>0$ is fixed. Thus the posterior density is proportional to $\exp\{-\tau_N\Psinf(\vec f_N)\}$. 

The scaling $\kappa_N/\tau_N=\text{const}$ keeps both the data term and the quadratic regularizer visible in the limiting continuum cost functional. The regularization term must remain visible in the limit $N\to\infty$, since the unregularized equation $\CK f=g$ may be severely ill-posed and may even fail to have a unique solution. 

We largely retain the notation introduced in the preceding sections. The only difference from the setting of Section~\ref{sec:constr BIP} is that the parameters $\tau=\tau_N$ and $\kappa=\kappa_N$ now depend on $N$.

The discrete MAP estimate $\vec u_{*,N}$ is the minimizer of \eqref{PsiN def}.  The expected continuum MAP estimate $u_*$ is obtained by minimizing the following functional over $f\in H^1$:
\bs\label{bip continuum action}
\Psif(f):=&\Psi(f)+\frac\rho2\int_{S^1}\big(f^2(t)+|f'(t)|^2\big)\dd t,\quad \Psi(f):=\int_{S^1}\Phi\big((\CK f)(t),g(t),t\big)\,\dd t.
\es
Similarly to \eqref{inv EL weak}, $u_*$ satisfies (see Lemma~\ref{lem:conv minim} below):
\be\label{inv EL weak L2}
\CK^*\big[\pa_1\Phi(\CK u_*,g,\cdot)\big]+\rho(-\pa_t^2+\Id)u_*=0\quad\text{in }H^{-1}.
\ee
 
Define the norm (see \eqref{PsiN def})
\bs\label{h1norm}
\|\vec f_N\|_{h_N^1}^2=\frac1{N}\sum_{i=0}^{N-1}\big[f_i^2+(D_N\vec f_N)_i^2\big],\quad \vec f_N\in\br^N.
\es

\begin{lemma}\label{lem:recon props}
For each $f\in H^1$ one has:
\begin{enumerate}
\item\label{weak ineq} If $\I_N\vec f_N\to f$ in $L^2$, then $\|f\|_{H^1}\le \liminf_{N\to\infty} \|\vec f_N\|_{h_N^1}$.
\item\label{cons eq} The sequence $\vec f_N=(N\I_N)^*f$ satisfies $\|\I_N\vec f_N-f\|_{L^2}\to0$ and $\|f\|_{H^1}= \lim_{N\to\infty} \|\vec f_N\|_{h_N^1}$.
\end{enumerate}
\end{lemma}

The lemma is proved in Appendix~\ref{sec:recon props}.

\begin{lemma}\label{lem:conv minim} The minimization problems
\be
\vec u_{*,N}:=\mathrm{arg\,min}_{\vec f\in \br^N} \Psinf(\vec f),\quad 
u_*:=\mathrm{arg\,min}_{f\in H^1} \Psif(f),
\ee
have unique solutions and $\I_N \vec u_{*,N}\to u_*$ in $L^2$. Moreover, $u_*$ satisfies \eqref{inv EL weak L2}.
\end{lemma}

\begin{proof}
Earlier, we showed that the functionals $\Psinf$, $N\in\N$, are coercive and strongly convex and therefore possess unique minimizers; see the paragraph preceding \eqref{inv EL weak}. We now establish the corresponding result for $\Psif$. The argument is somewhat similar.

The data term $\Psi(f)$ in \eqref{bip continuum action} is convex, continuous, and bounded from below, while the regularization term $(\rho/2)\|f\|_{H^1}^2$ is strongly convex and coercive on $H^1$. Consequently, $\Psif$ is proper, strongly convex, coercive, and continuous on $H^1$. Then the minimizer $u_*$ exists and is unique \cite[Theorem 1.3.1]{Ceg2012}, \cite[Corollary 2.20]{Peyp2015}.

Finally, the continuity of $\pa_1\Phi$ (Assumption~\ref{ass:setting}\eqref{Phi lip}) implies that $\Psif$ is differentiable. Hence $u_*$ satisfies the first-order optimality condition $D\Psif(u_*)[h]=0$, $h\in H^1$,
which is precisely \eqref{inv EL weak L2}; see, for example, \cite[Proposition 3.20 and Theorem 3.24]{Peyp2015}.

By Lemma~\ref{lem:recon props}\eqref{cons eq}, choose $\vec f_N$ such that $\I_N\vec f_N\to u_*$ in $L^2$ and $\|\vec f_N\|_{h^1_N}\to \|u_*\|_{H^1}$. By Lemma~\ref{lem:PsiNPsi}, $\Psi_N(\vec f_N)\to \Psi(u_*)$. Since $\vec u_{*,N}$ minimizes $\Psinf$,
$\Psinf(\vec u_{*,N})\le \Psinf(\vec f_N)$. Therefore 
\be\label{PsiN bnd}
\limsup_{N\to\infty}\Psinf(\vec u_{*,N})\le \Psif(u_*).
\ee

Recall that for a bounded variation function we have
\be
\|f\|_{\mathrm{BV}}:=\int_{S^1}|f(t)|\dd t+\mathrm{TV}(f),\quad
\TV(f):=\sup_{\substack{\varphi\in C^1,\|\varphi\|_{L^\infty}\le 1}}
\int_{S^1} f(t)\varphi'(t)\,\dd t.
\ee
Applying this to $f=\I_N\vec f_N$ we get
\be
\|f\|_{\mathrm{BV}}=\frac1N\sum_{i=0}^{N-1}|f_i|+\sum_{i=0}^{N-1}|f_{i+1}-f_i|
\le \sqrt2\|\vec f_N\|_{h_N^1}.
\ee

Denote $u_{*,N}:=\I_N\vec u_{*,N}$. By \eqref{PsiN bnd}, $\|\vec u_{*,N}\|_{h_N^1}\les 1$ uniformly in $N$. A bounded set in $\mathrm{BV}(S^1)$ is precompact in $L^2$ \cite[Proposition 4.1(5)]{bergounioux2017}. Take a subsequence such that $u_{*,N_k}$ converges to some $h$ in $L^2$. As above, the likelihood term converges along this subsequence: $\Psi_N(\vec u_{*,N_k})\to \Psi(h)$.
By Lemma~\ref{lem:recon props}\eqref{weak ineq},
\be
\|h\|_{H^1}\le \liminf_{N_k\to\infty}\|\vec u_{*,N_k}\|_{h^1_{N_k}}.
\ee
Hence
\be
\Psif(h) 
\le \liminf_{N_k\to\infty}\Psinf(\vec u_{*,N_k})
\le \limsup_{N_k\to\infty}\Psinf(\vec u_{*,N_k})
\le \Psif(u_*).
\ee
Since $u_*$ is the unique minimizer of $\Psif$, we get $h=u_*$. Thus every subsequence of $u_{*,N}$ has an $L^2$-convergent further subsequence and its limit is $u_*$. Therefore $u_{*,N}\to u_*$ in $L^2$.
\end{proof}

We expect that, to zeroth order, the posterior first concentrates on $u_*$, and the next order is governed by the Gaussian fluctuation field associated with the Hessian of \eqref{PsiN def} at $\vec u_{*,N}$. This is established in the next sections. But first we remind the reader of the result of \cite{katsevich2025c} and \cite{katsevich2026a} that we need here.

\subsection{High-dimensional Laplace asymptotics}\label{sec:high-d lapl} 

The results recalled in this subsection are independent of Assumptions~\ref{ass:det apert} and~\ref{ass:setting}.

We begin with the asymptotics of the integral
\be\label{main int v0}
I(\la)=\left(\frac{\la}{2\pi}\right)^{d/2}\int_{\br^d}e^{-\la f(\vec x)}\dd \vec x
\ee
as $d,\la\to\infty$. Denote $\e:=d/\la$ and $\|\cdot\|:=\|\cdot\|_{\ell^2}$. We make the following assumption on $f$.

\begin{assumption}\label{ass:f}$\hspace{1cm}$

\noindent
For some constants $c_f,r_0>0$ and $\varkappa>0$ independent of $d$ and $\la$ one has
\begin{align}\label{eq:Taylor-f}
&\big|f(\vec x)-\tfrac12 \norm{\vec x}^2\big|\le c_f\norm{\vec x}^3,\quad \norm{\vec x}\le r_0,\\
\label{f-suff}
&f(\vec x)\ge \varkappa\min\left\{\|\vec x\|^2/2, \sqrt{\e}\norm{\vec x}\right\},\quad \vec x\in\br^d.
\end{align}
\end{assumption}

The following result is proved in Appendix~\ref{sec:tail bnd proof} using the approach in \cite{katsevich2026a}.
\begin{lemma}\label{lem:tail bnd} Pick any $a>0$. Suppose $f$ satisfies Assumption~\ref{ass:f}. Let $\vec\xi\in \br^d$ be the random vector with the law $\propto e^{-\la f(\vec x)}$. There exist $R$ sufficiently large and $\de>0$ sufficiently small such that for any $d,\la$ satisfying $d/\la\leq \de$, 
it holds 
\be\label{tail bnd}
\BP\big(\{\|\vec\xi\|\ge R\e^{1/2}\}\big)\les \exp\{-ad\}.
\ee
\end{lemma}

The following is Theorem 2.2 of \cite{katsevich2025c}, which we restate here in a form convenient for our purposes. In the notation of \cite{katsevich2025c}, we set $\hat\theta=\vec 0$, $\Theta=\br^d$ and $H=I_d$.

\begin{theorem}\label{thm:meas} Let $f\in C^2(\br^d)$ be convex with unique strict global minimizer $\vec 0$ and $\nabla^2 f(\vec 0)=I_d$. Suppose there exists $r_0>0$ such that 
\be
\|\nabla^2 f(\vec x)-I_d\|_{\mathrm{op}}\les \|\vec x\|,\quad \|\vec x\|\le r_0. 
\ee
Then, for all $d$ and $\la$ so that $d/\la$ and $\log\la/d$ are sufficiently small, it holds
\be\label{main res meas}
\mathrm{TV}(\pi,\ga)=\tfrac12\sup_{\|g\|_{L^\infty(\br^d)}\leq1}\left|\E_{X\sim\pi}[g(X)]-\E_{X\sim\ga}[g(X)]\right|\les\frac{d}{\la^{1/2}},
\ee 
where $\pi\propto e^{-\la f(\vec x)}$ and $\ga=\mathcal \CN(\vec 0, \la^{-1}I_d)$. 
\end{theorem}

\subsection{Posterior concentration}\label{ssec:post conc}
Let $\vec \eta=\vec f_N-\vec u_{*,N}$ be the fluctuation variable. The corresponding fluctuation measure becomes
\be\label{bip fluc measure}
\dd\mu_N^d(\vec \eta)
=(Z_N^d)^{-1}\exp\{-\tau_N\Psinf(\vec u_{*,N}+\vec \eta)\}\,\dd\vec \eta,
\ee
where $\Psinf$ is defined in \eqref{PsiN def} and $Z_N^d$ is the normalizing constant. Throughout this section, the superscript $d$ labels measures and normalizing constants associated with the direct coupled limit described at the beginning of Section~\ref{ssec:diag lim setting}. The reconstructed posterior field is
\be
f_{N}:=\I_N \vec f_N=\I_N(\vec u_{*,N}+\vec \eta),
\quad \vec \eta\sim\mu_N^d,
\ee
and its law is the pushforward measure
\be
\nu_N^d:=(\vec \eta\to\I_N(\vec u_{*,N}+\vec \eta))_\#\mu_N^d.
\ee

\begin{lemma} One has
$\nu_N^d\Rightarrow \delta_{u_*}$ weakly in $L^2$ provided $\tau_N/N\to\infty$.
\end{lemma}

\begin{proof}
As is easily checked, $\|\I_N\vec \eta\|_{L^2}=N^{-1/2}\|\vec \eta\|_{\ell^2}$. By Lemma~\ref{lem:QNAN BIP}, $\|A_{N}\|_{\ell^2\to\ell^2}\les N^{1/2}$. Hence, after whitening $\vec \eta=A_{N}\vec x$, we have $\|\I_N\vec \eta\|_{L^2} \les \|\vec x\|_{\ell^2}$.


We now verify that the function $\psi_N(\vec x)$ defined in \eqref{inv whitened exponent} satisfies Assumption~\ref{ass:f}. Its definition in the present section is the same as in \eqref{inv whitened exponent}, but with $\Psinf$ now given by \eqref{PsiN def} rather than by \eqref{discr fnls}. Since the functional in \eqref{PsiN def} does not contain the overall factor $\tau_N$, the functional in \eqref{PsiN def} is obtained from the one in \eqref{discr fnls} by setting $\tau=1$ and replacing $\kappa$ with $\rho$. Consequently, all the properties of $\psi_N$ established in Section~\ref{sec:gen estims} continue to hold in the present setting.

By Lemma~\ref{lem:psi ders bdd} and \eqref{at minimizer}, $\psi_N$ satisfies \eqref{eq:Taylor-f} in Assumption~\ref{ass:f}. Combining these facts with Assumption~\ref{ass:setting}\eqref{bmin}, we conclude that $\psi_N$ is convex and that its Hessian is uniformly positive definite in a fixed-radius neighborhood of the origin. Hence $\psi_N(\vec x)$ satisfies \eqref{f-suff} in Assumption~\ref{ass:f}. 

Indeed, near the origin the quadratic term in the minimum on the right-hand side of \eqref{f-suff} is smaller than the linear term, whereas farther from the origin the linear term is smaller. The two terms are comparable when $\|\vec x\|_{\ell^2}\sim\sqrt\e$. At this crossover scale, the radial derivative of the quadratic lower bound is of order $\sqrt\e$. Convexity then implies at least linear growth beyond this scale, with slope of order $\sqrt\e$. This proves \eqref{f-suff}.

Lemma~\ref{lem:tail bnd} gives with $d=N$ and $\la=\tau_N$
\be
\BP\big(\big\{\|\vec x\|_{\ell^2}\le R(N/\tau_N)^{1/2}\big\}\big)\to 1
\ee
for a sufficiently large fixed $R$. Hence $\I_N\vec \eta\to0$ in probability in $L^2$ provided $N/\tau_N\to0$. Furthermore, Lemma~\ref{lem:conv minim} gives convergence of the deterministic centers: $\I_N\vec u_{*,N}\to u_*$ in $L^2$. Therefore
\be
f_{N}=\I_N\vec u_{*,N}+\I_N\vec \eta \to u_*
\quad\text{in probability in }L^2.
\ee
Convergence in probability to a deterministic limit implies weak convergence
of the corresponding laws to the delta function at that limit. 
\end{proof}

\subsection{Gaussian scale}\label{ssec:Gaus scale}
We now pass to the Gaussian scale.  Define the rescaled centered fluctuation
function
\be\label{bip rescaled fluct field}
\xi_{N}:=\sqrt{\tau_N}\,\I_N\vec \eta, \quad \vec \eta\sim\mu_N^d,
\ee
where $\mu_N^d$ is defined in \eqref{bip fluc measure}. Let $\ga_N:=\CN(\vec 0,Q_N^{-1})$ be the centered Gaussian on $\br^N$, where $Q_N$ is the Hessian of $\Psinf$ at $\vec u_{*,N}$:
\bs\label{bip discrete Hessian}
\langle \vec \eta,Q_N\vec \eta\rangle_{\ell^2}
=\textstyle
&\tfrac1{N}\tsum_{i=0}^{N-1}
b_{i}(\CK_{N}\vec \eta)_i^2
+\frac{\rho}N\sum_{i=0}^{N-1}\big[\eta_i^2+(D_N\vec \eta)_i^2\big],\\
b_{i}:=&\pa_1^2\Phi\big((\CK_{N}\vec u_{*,N})_i,g_i,t_i\big).
\es

In the preceding subsection we verified that $\psi_N(\vec x)$ is convex, has a global minimum at $\vec0$, $\nabla^2\psi_N(\vec 0)=I_N$, and $\|\nabla^3\psi_N(\vec x)\|_{\mathrm{op}}$ is uniformly bounded in $N$ and $\vec x$ in a fixed-radius neighborhood of $\vec0$. 

Applying Theorem~\ref{thm:meas} to the probability measure with density proportional to $\exp\{-\tau_N\psi_N(\vec x)\}$, and then using the invariance of total variation distance under the invertible linear transformation $\vec x\mapsto\sqrt{\tau_N}A_N\vec x$, gives
\be\label{bip tv gaussian approx}
\big\|\check\mu_N^d-\gamma_N\big\|_{\TV}\to0,
\ee
where $\check\mu_N^d$ is the law of $\sqrt{\tau_N}\vec\eta\in\br^N$. Using the above identification $d=N$ and $\la=\tau_N$, this holds if $N^2/\tau_N\to0$ and $\log\tau_N/N\to0$. By monotonicity of total variation under pushforward,
\be\label{bip pushforward tv}
\big\|(\I_N)_\#\check\mu_N^d-(\I_N)_\#\ga_N\big\|_{\TV}\to0.
\ee
Thus the true rescaled fluctuation field and the reconstructed Gaussian field have asymptotically the same law. It remains only to identify the continuum limit of the Gaussian measures $(\I_N)_\#\ga_N$.

As is easily seen, the corresponding reconstructed covariance is $\I_N Q_N^{-1} \I_N^*:L^2\to L^2$. The expected continuum covariance is $Q^{-1}$, where 
\bs\label{bip continuum Hessian}
\langle \eta,Q \eta\rangle_{L^2}
=& \int_{S^1} b(t)[(\CK \eta)(t)]^2\dd t
+\rho\int_{S^1} \big[\eta^2(t)+(\eta'(t))^2\big]\dd t,\\
b(t):=&\pa_1^2\Phi\big((\CK u_*)(t),g(t),t\big).
\es
In operator notation,
\be\label{bip continuum Hessian operator}
Q=\CK^*\left[b(t)\,\CK\right]+\rho(-\pa_t^2+\Id):\,H^2\to L^2.
\ee
As usual, periodic boundary conditions are understood. 

\begin{lemma}\label{lem:conv of gaus meas} 
Suppose that the additional Assumption~\ref{ass:extra} holds. The reconstructed Gaussian measures converge weakly: $(\I_N)_\#\ga_N \Rightarrow \CN(\vec 0,Q^{-1})$ in $L^2$.
\end{lemma}

From \eqref{two priors} and \eqref{bip continuum Hessian operator}, $\rho\CL\le Q$, so $0\le Q^{-1}\le (\rho\CL)^{-1}$. Since $\CL^{-1}$ is trace class on $L^2$, so is $Q^{-1}$ and $\CN(\vec 0,Q^{-1})$ is a Gaussian probability measure on $L^2$. The lemma is proved in Appendix~\ref{sec:conv gaus}. As is well-known, TV convergence implies weak convergence. Hence, combining \eqref{bip pushforward tv} with Lemma~ \ref{lem:conv of gaus meas} proves the following result.

\begin{lemma}\label{lem:bip final gaussian fluct} Suppose that the additional Assumption~\ref{ass:extra} holds. The reconstructed fluctuation measures converge weakly: $(\I_N)_\#\check\mu_N^d \Rightarrow \CN(\vec 0,Q^{-1})$ in $L^2$ when $N^2/\tau_N\to0$ and $\log\tau_N/N\to0$.
\end{lemma}

Recall that $(\I_N)_\#\check\mu_N^d$ is the law of $\xi_N=\sqrt{\tau_N}\I_N\vec\eta$, see \eqref{bip rescaled fluct field}. The convergence results in this and the preceding section justify the following approximation:
\be\label{approx family}
\I_N\vec f_N\approx u_*+\tau_N^{-1/2}\xi,
\quad \xi\sim\CN(\vec 0,Q^{-1}),
\ee
provided $N$ is sufficiently large and the growth conditions in Lemma~\ref{lem:bip final gaussian fluct} are satisfied. This provides a convenient approximation to the entire family of discrete BIPs for large $N$. Note that \eqref{approx family} does not assert that the center approximation error $u_*-\I_N\vec u_{*,N}$ is $o(\tau_N^{-1/2})$ in any norm. The term $\tau_N^{-1/2}\xi$ describes the residual posterior uncertainty after the measurements have been incorporated. The discretization error in approximating the center may occur on a different scale.

\section{Compatibility of the constructive and coupled limits}\label{sec:Compatibility}

The coupled-limit theorem and the constructive theorem are compatible in the sense that the small-noise limit of the continuum posterior $\mu$ constructed in Section~\ref{sec:constr BIP}, see \eqref{two posters}, has the same MAP estimate $u_*$ and Gaussian fluctuation law $\CN(\vec 0,Q^{-1})$ as the coupled limit obtained in Section~\ref{sec:small-noise-high-reg}.

\subsection{Matching MAP estimate}
The first step is to introduce the large parameters $\tau_N$ and $\kappa_N$ into the
constructive BIP formulation in Section~\ref{sec:constr BIP} in the same way as in \eqref{PsiN def}. The analog of the continuum prior $\nu$ in \eqref{two priors} becomes a small-noise Gaussian measure 
\be\label{sn prior}
\nu_N^c=\CN\big(0,(\tau_N\rho\CL)^{-1}\big),\quad \CL=-\pa_t^2+\Id,\quad \tau_N\to\infty.
\ee
The continuum posterior becomes
\be\label{cont post}
\dd\mu_N^c(f)=(Z_N^c)^{-1} \exp\{-\tau_N\Psi(f)\}\dd\nu_N^c(f),\quad
Z_N^c=\int_{L^2}\exp\{-\tau_N\Psi(f)\}\dd\nu_N^c(f),
\ee
where $\Psi$ is defined in \eqref{Psi def} with $\tau=1$. 

Throughout this section, the superscript $c$ indicates that the corresponding measures are defined on the continuum state space, rather than on $\br^N$, as the subscript $N$ might otherwise suggest.

Viewed with speed $\tau_N$, the Gaussian prior \eqref{sn prior} satisfies a small-noise LDP whose rate function is the Cameron--Martin energy $R_{\mathrm{CM}}(f)$ given in \eqref{PsiRCM} below \cite[Section 4.9, Corollary 4.9.3]{bogachev1998}. Clearly, this is a good rate function: the sets $\{f\in L^2: R_{\mathrm{CM}}(f)\le a\}$ are compact for any $a\in\br$. 

By Assumption~\ref{ass:setting}\eqref{bmin}, $-\Psi$ is bounded above. Therefore the tail condition in \cite[eq.~(4.3.2)]{dembo2010} is automatically satisfied. Moreover, $\Psi:L^2\to\br$ is continuous by Assumption~\ref{ass:setting}\eqref{Phi lip} and the continuity of $\CK:L^2\to C$. Hence, by \cite[Theorem~4.3.1 and Exercise~4.3.11]{dembo2010}, the exponentially tilted measures $\mu_N^c$ satisfy an LDP with good rate function

\be
I(f)=\Psif(f)-\inf_{L^2} \Psif,
\ee
where
\be\label{PsiRCM}
\Psif(f)=\Psi(f)+R_{\mathrm{CM}}(f),\quad 
R_{\mathrm{CM}}(f):=\begin{cases}\tfrac{\rho}2\|f\|_{H^1}^2,& f\in H^1,\\
\infty,& f\not\in H^1,
\end{cases}
\ee
and $R_{\mathrm{CM}}$ denotes the Cameron--Martin regularization term. 
Moreover, the functional $\Psif$ in \eqref{PsiRCM} is exactly the same as in the coupled-limit case \eqref{bip continuum action}.

By Lemma~\ref{lem:conv minim}, $\Psif$ has a unique minimizer $u_*$, so this LDP implies $\mu_N^c \Rightarrow \delta_{u_*}$ weakly on $L^2$. Consequently, the MAP profile obtained from the constructive continuum posterior agrees with the MAP profile obtained in the coupled limit.

\subsection{The Gaussian matching step on the constructive side}
\label{ssec:constructive-gaussian-matching}

We next prove that, after centering at $u_*$ and rescaling by $\sqrt{\tau_N}$, the constructive continuum posterior $\mu_N^c$ has the same Gaussian small-noise limit
$\CN(\vec 0,Q^{-1})$ as the coupled-limit posterior. More precisely, we prove that
\be
\pi_N^c:=\mathcal L_{\mu_N^c}
\left(\sqrt{\tau_N}(f-u_*)\right) \Rightarrow \CN(\vec 0,Q^{-1}),
\ee
where $Q$ is defined in \eqref{bip continuum Hessian}, \eqref{bip continuum Hessian operator}.

Set $\gamma=\CN(\vec 0,(\rho\CL)^{-1})$. By the Cameron--Martin theorem applied to the prior \cite[Corollary 2.4.3 and Theorem 2.4.5]{bogachev1998}, the fluctuation posterior $\pi_N^c$ has density
\be
\frac{d\pi_N^c}{d\gamma}(h)
\propto \exp\big\{-\sqrt{\tau_N}\,\ell_{u_*}(h)-\tau_N\Psi(u_*+\tau_N^{-1/2}h)\big\},
\ee
where $\ell_{u_*}$ denotes the Cameron--Martin linear functional associated with $u_*$ (see \cite[Theorem 2.8]{daprato2006}). The functional admits the representation
\be
\ell_{u_*}(h)=\langle \rho\CL u_*,h\rangle_{L^2},\quad h\in L^2.
\ee
Here we have used that $u_*\in H^2$, which follows from \eqref{inv EL weak L2} because $\CK^*:L^2\to C$ is bounded (see Assumption~\ref{ass:setting}\eqref{Kbdd}) and the first term on the left in \eqref{inv EL weak L2} is a continuous function. In fact, only the inclusion in $L^2$ would suffice here. Furthermore,
\be\label{interm deriv}
\begin{aligned}
\frac{d\pi_N^c}{d\gamma}(h)
\propto &\exp\big\{-\sqrt{\tau_N}
\left[\ell_{u_*}(h)+ D\Psi(u_*)[h]\right] -\tau_N B_N(h)\big\},
\end{aligned}
\ee
where 
\be\label{BN def}
B_N(h):=\Psi(u_*+\tau_N^{-1/2}h)-\Psi(u_*)-\tau_N^{-1/2}D\Psi(u_*)[h].
\ee

By convexity of $\Psi$, $B_N(h)\ge0$. The term in brackets in \eqref{interm deriv} vanishes by the Euler--Lagrange equation for $u_*$ (see \eqref{inv EL weak L2}). Thus
\be
\frac{d\pi_N^c}{d\gamma}(h)
\propto \exp\left\{-\tau_N B_N(h)\right\}.
\ee
From \eqref{BN def} and the continuity of $\pa_1^2\Phi$ (Assumption~\ref{ass:setting}\eqref{Phi lip}) it follows that
\be
F_N(h):=\exp\{-\tau_N B_N(h)\}\to F(h):=\exp\big\{-\tfrac{1}{2}D^2\Psi(u_*)[h,h]\big\}
\ee
for any $h\in L^2$. Therefore, by the dominated convergence theorem, 
\bs
\hat Z_N^c:=\int_{L^2}F_N(h) \dd \ga(h)\to &
Z^c:=\int_{L^2}F(h)\dd \ga(h),\\
(\hat Z_N^c)^{-1}\int_{L^2}G(h)F_N(h) \dd \ga(h)
\to & (Z^c)^{-1}\int_{L^2}G(h)F(h) \dd \ga(h),
\es
for any bounded continuous functional $G$. Therefore $\pi_N^c\Rightarrow \pi$ in $L^2$, where
\bs
&\frac{d\pi}{d\gamma}(h)
=(Z^c)^{-1}
\exp\left\{-\tfrac{1}{2}D^2\Psi(u_*)[h,h]\right\},\\
&D^2\Psi(u_*)[h,h]=\int b(t)(\CK h(t))^2\dd t.
\es
Since $\pi$ is a quadratic tilt of the Gaussian measure $\gamma=\CN(\vec 0,(\rho\CL)^{-1}\big)$, it is Gaussian with precision operator 
\be
Q=\CK^*\left[b(t)\,\CK\right]+\rho\CL, 
\ee
thereby matching the coupled Gaussian limit \eqref{bip continuum Hessian}, \eqref{bip continuum Hessian operator}.

\appendix

\section{Proof of Lemma~\ref{lem:conv of priors}}\label{sec:conv prior}

\subsection{Tightness of the Gaussian measures $\widetilde\nu_N$ in $L^2$}\label{ssec:Gaussian tightness}

For simplicity, we assume $\kappa=1$. The proof for the general $\kappa>0$ is the same. Denote $e(t):=\exp(2\pi i t)$. For the Fourier calculations in this section, we work in the complexifications of $L^2$ and $\br^N$, extending all real operators complex-linearly.

Let $e_k(x)=e(kx)$, $k\in\BZ$, be the Fourier basis of the fluctuation space $L^2$. Define the operator $T:L^2\to L^2$ by its action on the basis:
\be\label{Toper def}
T e_k=(1+|k|)^{-\al}e_k,\quad k\in \BZ,
\ee
for some $0<\al<1/2$. Clearly, $T$ is injective, nonnegative, compact and $\widetilde\nu_N(T(L^2))=1$ for all $N$.

By the properties of Gaussian measures in Hilbert space (see the proof of \cite[Example 3.7.10, part (ii)]{bogachev1998}) and using that $\widetilde\nu_N:=(\I_N)_\#\nu_N$,
\bs\label{T-est}
T_N:=&\int_{L^2} \|T^{-1}x\|_{L^2}^2\widetilde\nu_N(\dd x)=\sum_{k\in\BZ} (1+|k|)^{2\al}\langle \I_N L_N^{-1} \I_N^* e_k,e_k\rangle_{L^2}\\
=&\sum_{k\in\BZ} (1+|k|)^{2\al}\langle L_N^{-1} \I_N^* e_k,\I_N^* e_k\rangle_{\ell^2}.
\es

The operator $L_N:\mathbb C^N\to\mathbb C^N$ introduced in \eqref{LN def} is diagonal in the periodic discrete Fourier basis 
\be\label{discr F basis}
\vec e_k:=\big(1, e(k/N),\dots e(k(N-1)/N)\big)\in\mathbb C^N,\quad k=0,1,\dots N-1. 
\ee
Its eigenvalues are
\be\label{LN evals}
\mu_k(L_N)=4N\sin^2(\pi k/N)+(1/N),\quad k=0,\ldots,N-1.
\ee
Hence, if $k_N:=\mathrm{dist}(k,N\BZ)$, then $\mu_k(L_N)\gtrsim (1+k_N^2)/N$, $k=0,\ldots,N-1$. Furthermore,
\be\label{two bases}
(\I_N^* e_k)_j=\int_{j/N}^{(j+1)/N}e(kx)\dd x=\frac{c_{k,N}}N e(jk/N),\ j=0,\dots,N-1,\quad c_{k,N}:=\frac{e(k/N)-1}{2\pi i (k/N)}.
\ee
Thus, $\I_N^* e_k$ is an eigenvector of $L_N$ corresponding to the eigenvalue $\mu_k(L_N)$.

Substitute \eqref{two bases} into \eqref{T-est}. Computing the following quantities 
\be\label{ckn bnd}
|c_{k,N}|^2= \frac{\sin^2(\pi k/N)}{(\pi k/N)^2},\quad \|\vec e_k\|_{\ell^2}^2=N,
\ee
gives
\bs\label{T-est v3}
T_N\lesssim& \sum_{k\in\BZ} \frac{(1+|k|)^{2\al}}{(1+|k_N|)^2}\frac{\sin^2(\pi k/N)}{(k/N)^2}.
\es
If $k=0$, the bound in \eqref{ckn bnd} (and the same factor in \eqref{T-est v3}) is computed by continuity.

Suppose $k=r+mN$, where $|r|\le N/2$. Consider first the case $m=0$. The contribution of these terms is bounded by (up to an absolute factor)
\bs\label{T-est v4}
\sum_{r=0}^{N/2} \frac{(1+r)^{2\al}}{(1+r)^2}<\infty.
\es
The terms with $m\not=0$ are bounded by (up to an absolute factor)
\bs\label{T-est v5}
\sum_{m\ge1}\sum_{r=1}^{N/2} \frac{(mN)^{2\al}}{r^2}\frac{\sin^2(\pi r/N)}{m^2}
\lesssim \sum_{m=1}^{\infty}m^{2(\al-1)} \frac{N^{2\al}}{N^2}N\to0,
\es
where we used that $\al<1/2$.

Thus, the sequence of $T_N$ is uniformly bounded, and by \cite[Example 3.8.13]{bogachev1998}, the family of measures $\widetilde\nu_N$ is tight.

\subsection{Convergence of covariance forms}\label{ssec:Gaussian observ}

Let $\phi,\psi\in C^\infty$. In this subsection we prove that
\be\label{conv observ}
\langle \I_N L_N^{-1}\I_N^*\phi,\psi\rangle_{L^2}
\to \langle \CL^{-1}\phi,\psi\rangle_{L^2}.
\ee
Expressing both sides in terms of the Fourier coefficients of $\phi$ and $\psi$ relative to the basis $e_k(x)$, \eqref{conv observ} is equivalent to
\bs
\sum_{k\in \kind}
\frac{\left(\sum_{\ell\in\BZ}
c_{k+\ell N,N}\widetilde\phi_{k+\ell N}\right)
\overline{\left(\sum_{q\in\BZ}c_{k+qN,N}\widetilde\psi_{k+qN}\right)}}
{1+4N^2\sin^2\left(\frac{\pi k}{N}\right)}
\to
\sum_{k\in\BZ}
\frac{\widetilde\phi_k\overline{\widetilde\psi_k}}
{1+4\pi^2k^2},
\es
where $\kind=\{\lfloor -N/2\rfloor,\dots,\lfloor N/2-1\rfloor\}$. Given that $\widetilde\phi_k,\widetilde\psi_k=O(|k|^{-a})$, $k\to\infty$, for any $a>0$, this is equivalent to proving
\bs
\sum_{k\in \kind}
\frac{|c_{k,N}|^2\widetilde\phi_k \overline{\widetilde\psi_k}}
{1+4N^2\sin^2\left(\frac{\pi k}{N}\right)}+O(1/N)
\to 
\sum_{k\in\BZ}\frac{\widetilde\phi_k\overline{\widetilde\psi_k}}
{1+4\pi^2k^2},
\es
and the desired assertion is immediate from dominated convergence.

\subsection{Identification of the Gaussian limit}\label{ssec:id G lim}

Consider the operator $C_N=\I_N L_N^{-1} \I_N^*$ and denote $C_\infty=\CL^{-1}$ (see \eqref{two priors}). By construction,
\be\label{Cnorm bnd}
\|C_N\|_{L^2\to L^2} \le \|\I_N\|_{\br^N\to L^2}\|L_N^{-1}\|_{\br^N\to \br^N}\|\I_N^*\|_{L^2\to \br^N}.
\ee
We assume here, for example, that the $\ell^2$ norm is used in $\br^N$. By \eqref{LN evals}, $\|L_N^{-1}\|_{\br^N\to \br^N}\les N$. By the definition of $\I_N$,
\bs
\textstyle \|\I_N \vec f\|_{L^2}^2=\frac1N\sum_{k=0}^{N-1}f_k^2=\frac1N \|\vec f\|_{\ell^2}^2,
\es
so $\|\I_N\|_{\br^N\to L^2}=N^{-1/2}$. Thus $\|\I_N^*\|_{L^2\to \br^N}=N^{-1/2}$, and \eqref{Cnorm bnd} gives that the norms $\|C_N\|_{L^2\to L^2}$, $N\in\N$, are uniformly bounded. Thus, by density of smooth functions in $L^2$ and the uniform bound,
\be\label{gen conv}
\langle C_N\phi,\psi\rangle_{L^2}\to \langle C_\infty\phi,\psi\rangle_{L^2},\quad \forall \phi,\psi\in L^2.
\ee

Since $L^2$ is complete and separable and the Gaussian measures $\widetilde\nu_N$ are tight in $L^2$, every subsequence has a weakly convergent further subsequence by Prohorov's theorem \cite[Theorem 2.5.1]{Khoshnevisan2002}. Pick any convergent subsequence so that $\widetilde\nu_{N_k}\Rightarrow \mu$ for some $\mu$. By the definition of weak convergence, $\ell_\# \widetilde\nu_{N_k}\Rightarrow \ell_\#\mu$ for any $\ell\in (L^2)^*$. Let $\ell(x)=\langle x,h\rangle_{L^2}$ for some $h\in L^2$. 

By the continuity theorem \cite[Theorem 26.3]{billingsley2012}, the characteristic functions of $\ell_\# \widetilde\nu_{N_k}$, given by $s\to\exp\big\{-\langle C_{N_k} h,h\rangle_{L^2}s^2/2\big\}$ converge to that of $\ell_\#\mu$ pointwise. By \eqref{gen conv}, $\langle C_{N_k} h,h\rangle_{L^2}\to \langle C_\infty h,h\rangle_{L^2}$. Hence the characteristic function of $\ell_\#\mu$ is $\exp\big\{-\langle C_\infty h,h\rangle_{L^2}s^2/2\big\}$ and $\ell_\#\mu$ a centered Gaussian.

Since this holds for any $\ell\in (L^2)^*$, $\mu$ is Gaussian by definition. Then $\mu$ is determined by its covariance. By \eqref{gen conv}, its covariance is $C_\infty$ and $\mu=\CN(\vec 0,C_\infty)$. Since the family is tight, every subsequence has a weakly convergent further subsequence. The preceding argument shows that every such subsequential limit equals $\CN(\vec 0,C_\infty)$. Therefore $\widetilde\nu_N\Rightarrow\CN(\vec 0,C_\infty)$.

\section{Proof of Lemma~\ref{lem:recon props}}\label{sec:recon props}

We use the discrete and continuum bases $e_k(x)$ and $\vec e_k$ introduced in Appendix~\ref{ssec:Gaussian tightness}. As before, we denote $\kind=\{\lfloor -N/2\rfloor,\dots,\lfloor N/2-1\rfloor\}$. Writing
\be\label{f fN defs}
f(t)=\sum_{k\in \BZ}\tilde f_k e_k(t),\quad 
\vec f_N=\sum_{k\in\kind}\tilde h_{k,N} \vec e_k,
\ee
we have 
\bs
\|f\|_{H^1}^2=\sum_{k\in \BZ}(1+4\pi^2k^2)|\tilde f_k|^2,\quad
\|\vec f_N\|_{h_N^1}^2=\sum_{k\in\kind}\big(1+4\pi^2k^2|c_{k,N}|^2\big)|\tilde h_{k,N}|^2,
\es
where $c_{k,N}$ is defined in \eqref{two bases}. Moreover, the convergence $\I_N\vec f_N\to f$ in $L^2$ implies
\be\label{L2 conv coefs}
\sum_{k\in\BZ}|\tilde f_k-\overline{c}_{k,N}\tilde h_{k,N}|^2\to0.
\ee
Here we defined $\tilde h_{k,N}$, $k\in\BZ$, by periodicity: $\tilde h_{k+\ell N,N}=\tilde h_{k,N}$ for any $\ell\in\BZ$ and $k\in\kind$. Claim~\ref{weak ineq} is now obvious. Indeed, by \eqref{L2 conv coefs}, one has $\tilde h_{k,N}\to \tilde f_k$, $|k|\le K$, for each fixed $K>0$. Moreover, $c_{k,N}\to1$ for each $k$. The tail of the series giving $\|f\|_{H^1}^2$ that corresponds to $|k|>K$ can be made arbitrarily small by choosing $K$ large, while the corresponding tail in $\|\vec f_N\|_{h_N^1}^2$ is nonnegative.

To prove claim \ref{cons eq}, we first show that $\tfrac1N\|(N\I_N)^*f\|_{\ell^2}^2\to\|f\|_{L^2}^2$. A simple manipulation shows that
\be
\|f-\I_N(N\I_N)^*f\|_{L^2}^2=\|f\|_{L^2}^2-\tfrac1N\|(N\I_N)^*f\|_{\ell^2}^2.
\ee
By an easy calculation,
\be\label{th formula}
\tilde h_{k,N}=\sum_{\ell\in\BZ}c_{k+\ell N,N}\tilde f_{k+\ell N},\quad k\in\kind.
\ee
Then
\bs\label{th formula v2}
&\sum_{\ell\in\BZ\setminus\{0\}}|c_{k+\ell N,N}\tilde f_{k+\ell N}|
\les  \frac{A_{k,N}}{N},\quad k\in\kind,\\ 
&A_{k,N}^2:=F_{k,N}^2-(1+4\pi^2k^2)|\tilde f_k|^2,\\
&F_{k,N}^2:=\sum_{\ell\in\BZ}\big(1+4\pi^2(k+\ell N)^2\big)|\tilde f_{k+\ell N}|^2,
\es
where $A_{k,N}$ denotes the nonnegative square root. This is proved by multiplying and dividing each term in the sum on the first line of \eqref{th formula v2} by $\big(1+4\pi^2(k+\ell N)^2\big)^{1/2}$, and then applying the Cauchy--Schwarz inequality, the bound $|c_{m,N}|\le 1$, $m\in\BZ$, and the elementary estimate
\be
\sup_{k\in\kind}
\sum_{\ell\in\BZ\setminus\{0\}}
\frac{1}{1+4\pi^2(k+\ell N)^2}
\les N^{-2}.
\ee

Observe that
\bs\label{Ak prop}
\sum_{k\in\kind}A_{k,N}^2=&\sum_{k\in\kind} F_{k,N}^2-\sum_{k\in\kind} (1+4\pi^2k^2)|\tilde f_k|^2\\
=&\sum_{k\in\BZ\setminus\kind} (1+4\pi^2k^2)|\tilde f_k|^2\to0.
\es
Therefore,
\bs\label{for H1 conv}
\sum_{k\in\kind}|\tilde h_{k,N}|^2
=\sum_{k\in\kind}\Big(|c_{k,N}\tilde f_k|^2
+O\Big(|\tilde f_k|\frac{A_{k,N}}{N}\Big)+O\Big(\frac{A_{k,N}^2}{N^2}\Big)\Big).
\es
Given that $\sum_{k\in\kind} A_{k,N}^2<\infty$ uniformly in $N$, we get that $\tfrac1N\|(N\I_N)^*f\|_{\ell^2}^2\to\|f\|_{L^2}^2$.

Thus, to finish the proof of claim \ref{cons eq}, it remains to find the limit of the following sum:
\bs\label{last sum}
\sum_{k\in\kind}k^2|c_{k,N}\tilde h_{k,N}|^2
=\sum_{k\in\kind}\Big(k^2|c_{k,N}^2\tilde f_k|^2
+O\big(|k\tilde f_k|A_{k,N}\big)+O\big(A_{k,N}^2\big)\Big).
\es
Using \eqref{Ak prop} and that $f\in H^1$,
\bs\label{tail to 0}
\Big[\sum_{k\in\kind}|k\tilde f_k|A_{k,N}\Big]^2
\le \sum_{k\in\kind} |k\tilde f_k|^2\sum_{k\in\kind} A_{k,N}^2\to0.
\es
Combining \eqref{Ak prop}--\eqref{tail to 0} with dominated convergence for the main term in \eqref{last sum} proves $\|\vec f_N\|_{h_N^1}\to\|f\|_{H^1}$.

\section{Proof of Lemma~\ref{lem:tail bnd}}\label{sec:tail bnd proof}

Fix some $R>0$ and define the following local and tail integrals:
\be
I_{\mathrm{in}}(\la)=\left(\frac{\la}{2\pi}\right)^{d/2}\int_{\|\vec x\|\leq R\sqrt\e}e^{-\la f(\vec x)}\dd \vec x,\quad
I_{\mathrm{out}}(\la)=\left(\frac{\la}{2\pi}\right)^{d/2}\int_{\|\vec x\|>R\sqrt\e}e^{-\la f(\vec x)}\dd \vec x.\label{eq:in out}
\ee
Similarly to \cite[Section~7]{katsevich2026a}, we need an upper bound on $I_{\mathrm{out}}(\la)$ and a lower bound on $I_{\mathrm{in}}(\la)$. The lower bound on $I_{\mathrm{in}}(\la)$ can be taken directly from \cite{katsevich2026a}, because it is proven under Assumption~\ref{ass:f}. By \cite[eq. (7.2)]{katsevich2026a} we have
\be\label{inlb}
I_{\mathrm{in}}(\la)\geq \tfrac12 \mu^{-d/2},\quad \mu:=1+2c_f R\sqrt\e,
\ee 
where we have used that $f(\vec x)\le \mu \|\vec x\|^2/2$, $\|\vec x\|\leq R\sqrt\e$, by \eqref{eq:Taylor-f}.

Let
\bs
T_1&=\Big(\frac{\la}{2\pi}\Big)^{d/2}\int_{\norm{\vec x}\ge R\sqrt\e} e^{-\varkappa\sqrt{\la d}\|\vec x\|}\,\dd \vec x,\quad
T_2=\Big(\frac{\la}{2\pi}\Big)^{d/2}\int_{\norm{\vec x}\ge R\sqrt\e} e^{-\la\varkappa\|\vec x\|^2/2}\,\dd \vec x.
\es
By~\eqref{f-suff} and \eqref{eq:in out}, $I_{\mathrm{out}}(\la)\leq T_1+T_2$. By \cite[eq. (7.9) and (7.10)]{katsevich2026a},
\be\label{T2T3}
T_1 \leq e^{-a d},\quad  T_2\leq e^{-a d}
\ee
for any fixed $a>0$, where $R$ is sufficiently large. Therefore
\be
I_{\mathrm{out}}(\la)/I_{\mathrm{in}}(\la) \les \exp\{-ad\}
\ee
for any fixed $a>0$, assuming $\e>0$ is sufficiently small. This immediately implies \eqref{tail bnd}.

\section{Proof of Lemma~\ref{lem:conv of gaus meas}}\label{sec:conv gaus}

First, we establish tightness of the Gaussian measures $(\I_N)_\#\ga_N$ in $L^2$. We use the same Fourier basis $e_k(x)=e(kx)$ of the fluctuation space $L^2$ and the same operator $T:L^2\to L^2$ defined at the beginning of Appendix~\ref{ssec:Gaussian tightness}. As before, in the Fourier calculations in this appendix, we work in the complexifications of $L^2$ and $\br^N$, extending all real operators complex-linearly. Similarly to \eqref{T-est},
\bs\label{T-est v2}
T_N:=&\int_{L^2} \|T^{-1}x\|_{L^2}^2(\I_N)_\#\ga_N(\dd x)
=\sum_{k\in\BZ} (1+|k|)^{2\al}\langle Q_N^{-1} \I_N^* e_k,\I_N^* e_k\rangle_{\ell^2}.
\es
Comparing \eqref{bip discrete Hessian} with \eqref{LN prep}, \eqref{LN def}, we have $\langle \eta,L_N\eta\rangle_{\ell^2} \lesssim \langle \eta,Q_N\eta\rangle_{\ell^2}$, which implies
\be\label{lbnd v2}
\langle Q_N^{-1} \I_N^* e_k,\I_N^* e_k\rangle_{\ell^2}\les \langle L_N^{-1} \I_N^* e_k,\I_N^* e_k\rangle_{\ell^2}.
\ee
Arguing exactly as in the rest of the proof of Lemma~\ref{lem:conv of priors} we obtain that the sequence of $T_N$ is uniformly bounded and the family of measures $(\I_N)_\#\ga_N$ is tight.

Next we prove that
\be\label{conv observ v2}
\langle \I_N Q_N^{-1}\I_N^*\phi,\psi\rangle_{L^2}
\longrightarrow \langle Q^{-1}\phi,\psi\rangle_{L^2}
\ee
for any $\phi,\psi\in C^\infty$. Recall that $Q$ is defined in \eqref{bip continuum Hessian operator}. In this section we use the notation $\tilde\I_N^*=N\I_N^*$, $(\tilde\I_N^d)^*=(N\I_N^d)^*$, and $\tilde Q_N:=NQ_N$. Let $\vec h_N\in\br^N$ and $h\in H^2$ be solutions to the following equations
\be
\tilde Q_N\vec h_N= \tilde\I_N^*\phi,\quad Qh=\phi.
\ee
Analyzing the accuracy of the discrete approximation (see \eqref{bip discrete Hessian} and \eqref{bip continuum Hessian operator}) we get
\be\label{rvec bnd}
\tilde Q_N(\tilde\I_N^* h)=\tilde\I_N^*\phi-\vec r_N,\quad \|\vec r_N\|_{\ell^\infty}=o(1),
\ee
where $\vec r_N$ is the residual error. Indeed, the continuity of $\CK^*:L^2\to C$ (Assumption~\ref{ass:setting}\eqref{Kbdd}) and \eqref{bip continuum Hessian operator} imply that $h''\in C$. Hence, by standard analysis, the accuracy of the approximation of the regularization term is $o(1)$. 

Consider next the likelihood part of $\tilde Q_N$. By \eqref{bip discrete Hessian}, it can be written in the form
\bs\label{QN lklhd part}
\tilde Q_N^L:=\tilde\I_N^*\CK^*\I_N^d\mathrm{diag}(\vec b)(\tilde\I_N^d)^*\CK\I_N.
\es
In \eqref{rvec bnd} we apply $\tilde Q_N$ and, therefore, $\tilde Q_N^L$ to a vector that approximates pointwise values of a continuous function in the $\ell^\infty$ metric. By the stated properties of $\CK$, $\I_N$ and $\I_N^d$, we have
\be\label{interm conv res}
(\mathrm{diag}(\vec b)(\tilde\I_N^d)^*\CK\I_N)(\vec f_N^{(0)}+o(1))
=\vec{(b\CK f)}_N^{(0)}+o(1);\quad (\vec f_N^{(0)})_i:=f(t_i),\ i=0,1,\dots,N-1,
\ee
for some $f\in C$. Here we use the continuity of $\pa_1^2\Phi$ (Assumption~\ref{ass:setting}\eqref{Phi lip}) to obtain convergence of the corresponding discrete coefficients. 

Given that $\CK^*:L^2\to C$ is continuous, the key remaining step is to show that $\I_N^d \vec f_N^{(0)} \to f$ in a suitable sense in $L^2$ for any $f\in C$. This is the place where the background Assumptions~\ref{ass:det apert} and \ref{ass:setting} are not sufficient and we need to impose Assumption~\ref{ass:extra}. It is easy to see that the desired property holds when either one of the conditions in Assumption~\ref{ass:extra} holds. Indeed, if Assumption~\ref{ass:extra}\eqref{apert extra} holds, then $\I_N^d \vec f_N^{(0)} \to f$ strongly in $L^2$. If Assumption~\ref{ass:extra}\eqref{K extra} holds, then we only have weak convergence in $L^2$, but the compactness of $\CK^*:L^2\to C$ gives the desired strong convergence $\CK^*\I_N^d \vec f_N^{(0)} \to \CK^* f$ in $C$. The outermost operator $\tilde\I_N^*$ in \eqref{QN lklhd part} presents no difficulties because it is uniformly bounded from $C$ to $\ell^\infty$. The remaining $o(1)$ vector in \eqref{interm conv res} presents no difficulties as well, and we get the bound in \eqref{rvec bnd}.

Denoting the reconstruction error $\vec e_N=\vec h_N-\tilde\I_N^* h$, we have $\tilde Q_N\vec e_N=\vec r_N$. This implies
\bs\label{last part}
\langle \I_N Q_N^{-1}\I_N^*\phi,\psi\rangle_{L^2}
=\langle \tilde Q_N^{-1}\tilde\I_N^*\phi,\I_N^*\psi\rangle_{\ell^2}
=\langle \vec h_N,\I_N^*\psi\rangle_{\ell^2}=\tfrac1N\langle \tilde\I_N^* h,\tilde\I_N^*\psi\rangle_{\ell^2}
+\tfrac1N\langle \vec e_N,\tilde\I_N^*\psi\rangle_{\ell^2}.
\es

Applying Lemma~\ref{lem:QNAN BIP} with $\tau=1$ and $\kappa=\rho$, we obtain the bounds \eqref{eq:A-bounds} for the present operators $A_N$ and $Q_N$. Therefore, by \eqref{rvec bnd},
\be
\|\vec e_N\|_{\ell^\infty}\les \tfrac1N \|Q_N^{-1/2}\vec r_N\|_{\ell^2}
\les \tfrac1N N^{1/2}\|\vec r_N\|_{\ell^2}=o(1).
\ee
Hence the last term on the far right in \eqref{last part} goes to zero. Clearly, the first term on the far right in \eqref{last part} goes to $\langle h,\psi\rangle_{L^2}$, and \eqref{conv observ v2} is proved.

Using that $\rho\CL_N\le Q_N$, the argument in Appendix~\ref{ssec:id G lim} implies that the norms $\|\I_N Q_N^{-1}\I_N^*\|_{L^2\to L^2}$, $N\in\N$, are uniformly bounded. Consequently, by density, \eqref{conv observ v2} extends to all $\phi,\psi\in L^2$. Combining this convergence with tightness, the subsequence and characteristic-function argument of Appendix~\ref{ssec:id G lim} applies verbatim, and hence $(\I_N)_\# \ga_N \Rightarrow\CN(\vec 0,Q^{-1})$.

\bibliographystyle{abbrv}
\bibliography{My_Collection}
\end{document}